\documentclass{siamart220329}
\usepackage{caption}
\UseRawInputEncoding
\usepackage{algorithm}
\usepackage{algpseudocode}

\usepackage{comment}

\makeatletter
\def\algbackskip{\hskip-\ALG@thistlm}
\makeatother

\usepackage{amsmath}
\usepackage{mathtools}
\usepackage{amssymb}
\usepackage{mathrsfs}
\usepackage{bm}
\usepackage{algpseudocode}
\usepackage{upgreek}
\usepackage{subcaption}

\usepackage{graphicx}
\usepackage{epstopdf}

\usepackage{stmaryrd}
\usepackage[shortlabels]{enumitem}

\newcommand{\cc}{{\normalfont{\text{c}}}}
\newcommand{\loc}{{\normalfont{\text{loc}}}}

\newcommand{\br}{{\normalfont{\boldsymbol{r}}}}

\DeclareMathOperator*{\argmin}{arg\,min}
\DeclareMathOperator*{\trace}{trace}

\newcommand{\dd}{{\normalfont{\text{d}}}}

\usepackage{xcolor}
\definecolor{ao}{rgb}{0.0, 0.5, 0.0}

\usepackage{bbm}

\usepackage{hyperref}
\usepackage{cleveref}

\makeatletter
\newcommand*{\encircled}[1]{\relax\ifmmode\mathpalette\@encircled@math{#1}\else\@encircled{#1}\fi}
\newcommand*{\@encircled@math}[2]{\@encircled{$\m@th#1#2$}}
\newcommand*{\@encircled}[1]{%
  \tikz[baseline,anchor=base]{\node[draw,circle,outer sep=0pt,inner sep=.2ex] {#1};}}
\makeatother

\newcommand{\Cov}{\mathbb{C}\!\operatorname{ov}}

\DeclareSymbolFont{extraup}{U}{zavm}{m}{n}
\DeclareMathSymbol{\varheart}{\mathalpha}{extraup}{86}
\DeclareMathSymbol{\vardiamond}{\mathalpha}{extraup}{87}

\usepackage{tikz}

\usepackage{pifont}

\usepackage{fancyhdr}

\usepackage{color}
\newcommand{\half}{\frac{1}{2}}

\newcommand{\norm}[1]{\left \lVert #1 \right \rVert}
\newcommand{\snorm}[1]{\left \lvert #1 \right \rvert}

\newcommand{\IR}{\mathbb{R}}
\newcommand{\IC}{\mathbb{C}}

\newcommand{\IN}{\mathbb{N}}

\newcommand{{\diff}}[1]{{\normalfont{\text{d}} #1}}

\usepackage{tikz}
\usepackage{listofitems} %
\usetikzlibrary{arrows.meta} %
\usepackage[outline]{contour} %
\contourlength{1.4pt}

\tikzset{>=latex} %
\usepackage{xcolor}
\definecolor{myred}{RGB}{230,75,53}

\definecolor{myblue}{RGB}{77,187,213}

\definecolor{mygreen}{RGB}{0,160,135}

\colorlet{myorange}{orange!70!red!60!black}

\colorlet{mydarkred}{myred!30!black}
\colorlet{mydarkblue}{myblue!40!black}
\colorlet{mydarkgreen}{green!50!black}

\tikzstyle{node}=[thick,circle,draw=myblue,minimum size=22,inner sep=0.5,outer sep=0.6]
\tikzstyle{node in}=[node,green!20!black,draw=mygreen!30!black,fill=mygreen!40]
\tikzstyle{node hidden}=[node,blue!20!black,draw=myblue!30!black,fill=myblue!40]
\tikzstyle{node convol}=[node,orange!20!black,draw=myorange!30!black,fill=myorange!20]
\tikzstyle{node out}=[node,red!20!black,draw=myred!30!black,fill=myred!40]

\tikzstyle{connect}=[thick,black] %
\tikzstyle{connect arrow}=[-{Latex[length=4,width=3.5]},thick,mydarkblue,shorten <=0.5,shorten >=1]
\tikzset{ %
  node 1/.style={node in},
  node 2/.style={node hidden},
  node 3/.style={node out},
}
\def\nstyle{int(\lay<\Nnodlen?min(2,\lay):3)}

\newcommand{\OV}{\mathsf{V}}
\newcommand{\OK}{\mathsf{K}}

\newcommand{{\D}}{\normalfont{\text{D}}}
\newcommand{{\G}}{\normalfont{\text{G}}}
\newcommand{{\U}}{\normalfont{\mathbb{U}}}
\newcommand{{\I}}{\normalfont{\text{I}}}

\newcommand{{\A}}{\normalfont{\text{A}}}

\newcommand{\isdef}{\mathrel{\mathrel{\mathop:}=}}
\newcommand{\defis}{\mathrel{=\mathrel{\mathop:}}}

\newcommand{\bchi}{\boldsymbol\chi}

\newcommand{\bu}{{\bf  u}}

\newcommand{{\bx}}{{\bf x}}
\newcommand{{\bz}}{{\bf z}}
\newcommand{\bxref}{\widehat{\bx}}
\newcommand{\yref}{\widehat{\y}}

\newcommand{{\bvarphi}}{{\boldsymbol{\varphi}}}

\newcommand{\OA}{\mathsf{A}}
\newcommand{{\N}}{\normalfont{\text{N}}}
\newcommand{{\y}}{{\boldsymbol{{y}}}}
\newcommand{{\h}}{{\boldsymbol{{h}}}}
\newcommand{{\z}}{{\boldsymbol{z}}}

\newcommand{{\bc}}{{\bf c}}
\newcommand{{\bd}}{{\bf d}}

\usepackage{bm}

\newcommand{\dotp}[2]{\left ( #1,#2 \right )}

\theoremstyle{remark}
\newtheorem{remark}{Remark}
\newtheorem{assumption}[theorem]{Assumption}

\numberwithin{equation}{section}

\begin{document}

\title{Fully discrete analysis of the Galerkin POD neural network approximation with application to 3D acoustic wave scattering}

\headers{Fully discrete analysis of the Galerkin POD-NN}{J. D\"{o}lz and F. Henr\'iquez }

\author{
J\"{u}rgen D\"{o}lz\thanks{Institute for Numerical Simulation, University of Bonn, Friedrich-Hirzebruch-Allee 7, 53115 Bonn, Germany (\email{doelz@ins.uni-bonn.de}).}
\and 
Fernando Henr\'iquez\thanks{Institute for Analysis and Scientific Computing, Vienna University of Technology, Wiedner Hauptstra{\ss}e 8-10, A-1040 Wien, Austria. (\email{fernando.henriquez@asc.tuwien.ac.at}).
}}

\maketitle

\begin{abstract}
In this work, we consider the approximation of parametric maps
using the so-called Galerkin POD-NN method.
This technique combines the computation of a reduced basis via proper orthogonal decomposition (POD)
and artificial neural networks (NNs) for the construction of fast surrogates of said parametric maps.
In contrast to the existing literature, which has studied the approximation properties of this kind of architecture on a continuous level, we provide a fully discrete error analysis of this approach. More precisely, our estimates also account for discretization errors during the construction of the NN architecture. We consider
the number of reduced basis in the
approximation of the solution manifold, truncation in the parameter space, and, most importantly,
the number of samples in the computation of the reduced space, together with the effect of the use of NNs in the approximation of the reduced coefficients.
Following this error analysis, we provide a-priori bounds on the required POD tolerance, the resulting POD ranks, and NN parameters to maintain the order of convergence of quasi Monte Carlo sampling techniques.

We conclude this work by showcasing the applicability of this method through a practical industrial application: the sound-soft 
acoustic scattering problem by a parametrically defined scatterer in three physical dimensions.  
\end{abstract}

\section{Introduction}\label{sec:introduction}
\subsection{Motivation}
\emph{Surrogate models} are key ingredients for the success of many-query applications where expensive computational models need to be repeatedly evaluated.
Indeed, if the underlying model, obtained using standard discretization methods such as the Finite Element, Finite Volume, Finite Difference, or Boundary Element method (the latter for the numerical approximation of boundary integral equations, or BIEs), is computationally demanding, the repeated use of this model becomes cost-prohibitive.

A good surrogate model provides a fast approximation of the original problem, while
guaranteeing a certified level of accuracy with respect to the so-called \emph{high-fidelity} solution.
Motivated by partial differential equations with random input data, inverse problems, and optimal control, the subject of this article are surrogate models to the \emph{parameter-to-solution map}
arising from parametric PDEs and BIEs.

More precisely, given Hilbert spaces $\mathcal{X},\mathcal{Y}$
and a compact subset $\mathcal{U}\subset \mathcal{X}$, referred as the 
\emph{parameter space}, we consider the problem of finding for each in 
$\nu\in \mathcal{U}$ the solution $u(\nu)\in \mathcal{Y}$ to the following
problem cast in variational form
\begin{equation}\label{eq:motivationvarprob}
	\mathsf{a}(u(\nu),v;\nu)
	=
	\mathsf{f}(v;\nu),
	\quad
	\forall v \in \mathcal{Y},
\end{equation}
where $\mathsf{f}(\cdot;\nu) \in \mathcal{Y}'$, and $\mathsf{a}(\cdot,\cdot;\nu)$
represents in a general framework either a linear or non-linear PDE or BIE. 
Provided that \cref{eq:motivationvarprob} is well-posed for each input $\nu \in \mathcal{U}$
one can define the \emph{parameter-to-solution} map $\mathcal{U}\ni \nu \mapsto u(\nu) \in \mathcal{Y}$.

The construction of surrogate models to the (nonlinear) parameter-to-solution map given through \cref{eq:motivationvarprob} is obstructed by several issues. First, the infinite dimensionality of the parameter space requires appropriate truncation for numerical computations. Second, the variational problem can usually only be solved approximately by means of a Galerkin approach which often requires many degrees of freedom. And third, the solution manifold
\begin{equation}
	\mathcal{M}
	\coloneqq  
	\left\{u(\nu) \in \mathcal{Y}: \nu \in \mathcal{U} \right\}
\end{equation}
is itself infinite or high dimensional without any
further measures.

The recent success of deep learning techniques in several fields of
science and engineering has led to many works suggesting the use \emph{artificial neural networks} in the approximation of the
parameter-to-solution map, a growing field referred to as
\emph{Operator Learning} or \emph{Neural Operators}.

\subsection{Related work}\label{sec:introduction_related_work}
The so-called ``training'' of plain vanilla neural networks can be interpreted as solving a non-linear regression problem, i.e., to fit the parameters of a specific non-linear function $\mathbb{R}^N\to\mathbb{R}^M$ (in this case the neural network) such that a loss functional is minimized \cite{BGG2024,GK2022}. It is well known that neural networks allow for universal approximation of properties, that is, they can approximate certain function classes up to arbitrary accuracy as long as the network is wide enough; see, e.g., \cite{BGKP2019,Cyb1989,DHP2021,Hor1991,PV2018}. 
Many results provide guaranteed convergence rates \cite{de2022generic,DLM21,HOS22,SZ2019}, with a particular focus on beating the so-called \emph{curse of dimensionality} in the parameter space.

However, using NNs to approximate the parameter-to-solution map requires architectures that account for the high, or even infinite, dimensionality of both the input and output spaces. In the last years, several of these have been proposed, e.g. DeepONets
\cite{lu2021learning,lanthaler2022nonlinear,lanthaler2022error,wang2021learning},
Fourier Neural Operators \cite{guibas2021adaptive,kovachki2021universal,li2020fourier},
graph neural operators \cite{pichi2024graph}, DIPNets \cite{ODC+2022,OVCG2022},
convolutional neural operators \cite{franco2023approximation,raonic2023convolutional}, and PCA-nets \cite{BHKS2021,herrmann2024neural,HU18,lanthaler2023operator,SZ2019}.
The latter is also known as the Galerkin POD-NN, and is the focus of this work.
Though not the framework considered in this work, we mention in passing 
that the Galerkin POD-NN has been extended to time-dependent 
problems, see e.g.~\cite{guo2019data,wang2019non,wang2020recurrent}.

In order to appropriately deal with the parameter-to-solution map
$\mathcal{U} \ni \nu \mapsto u(\nu) \in \mathcal{Y}$ and effectively build a suitable method that leverages on NNs
for its approximation, the complexity of both the input
and output spaces needs to be appropriately described by a 
\emph{encoder-decoder} pair \cite{BHKS2021,FMZ2022,HU18,KPRS2022,ODC+2022}.
That is, one construct maps $\mathscr{E}: \mathcal{U} \rightarrow \mathbb{R}^N $, $\boldsymbol{\pi}: \mathbb{R}^N \rightarrow\mathbb{R}^M $, and
$\mathscr{D}: \mathbb{R}^M \rightarrow \mathcal{M}$ such that
\begin{equation}\label{eq:diagram}
	\mathcal{X}
	\supset
	\mathcal{U}
	\xrightarrow{\text{Encoder }\mathscr{E} }	
	\mathbb{R}^N
	\xrightarrow{\text{Neural Network }\boldsymbol{\pi}}	
	\mathbb{R}^M
	\xrightarrow{\text{Decoder }\mathscr{D}}
	\mathcal{M}
	\subset
	\mathcal{Y}.
\end{equation}
Then, for each input $\nu \in \mathcal{U}$ one can construct
an approximation of the parameter-to-solution map  which reads
as $u(\nu) \approx \left(\mathscr{D} \circ \boldsymbol{\pi} \circ \mathscr{E}\right) (\nu)$.
Indeed, the task of operator learning boils down to defining the encoder-decoder pair and
computing an appropriate NN $\boldsymbol{\pi}$. This NN is of hopefully moderate size.
This is understood not only in the sense that $N$ and $M$ are controlled, but also in terms of the overall network size, which is determined by the number of trainable parameters.

The Galerkin POD-NN method relies on the combination
of projection-based model order reduction techniques
for the construction of the decoder, in particular, the reduced basis method \cite{hesthaven2016certified,QMNMOR2016} and, of course, NNs.
The reduced basis method, which is at the core of the methodology presented here, follows a two phase paradigm--online
and offline--for the swift and efficient evaluation of the
parameter-to-solution map.
In the offline phase, a basis of
reduced dimension is computed by properly sampling the space $\mathcal{U}$ and performing
a proper orthogonal decomposition (POD), although \emph{greedy} strategies could also be put in place. 
These allow for the identification of the most important modes driving the dynamics of the parameter-to-solution map.
Next, in the online phase, the evaluation of the parameter-to-solution map for a given parametric input is computed in a variational fashion as an element of the reduced space. For this purpose,
hyper-reduction techniques, such as the empirical interpolation method \cite{barrault2004empirical} and its discrete counterpart
\cite{chaturantabut2010nonlinear}, can be used. However, these techniques are intrusive in nature, and their implementation is not trivial. Instead, the idea of the Galerkin-POD NN, originally pointed out in \cite{HU18}, is to use a NN for the approximation of the reduced coefficients,
i.e.~for the computation of the central part in \cref{eq:diagram}. This completely decouples the online and offline phases, and makes the approximation of the reduced coefficients purely \emph{data-driven}.

Summarized, the Galerkin-POD NN provides an \emph{algorithmically implementable} and \emph{computationally feasible} construction of the decoder. What remains to be understood is the interplay of the approximation errors of the neural network approximation, the Galerkin-POD, and the Galerkin approximation to the variational problem and how to balance them to obtain an accurate and computationally efficient approximation scheme.

\subsection{Contributions}
The goal of this paper is to go beyond NN approximation rates on a continuous level and to provide fully discrete a-priori approximation rates of the neural network approximation of parameter-to-solution maps to \cref{eq:motivationvarprob} which account for all approximation errors occurring in \emph{algorithmically feasible} computations. To this end, we consider the Galerkin POD-NN architecture when the parameter-to-solution map and its Galerkin approximation are $(\boldsymbol{b},p,\varepsilon)$-holomorphic and the samples are drawn according to quasi-Monte Carlo (QMC) rules, as for example suggeted in \cite{LMRS2021,LMR2020,MR2021}. For this setting we
\begin{enumerate}
	\item[(i)]
	Provide a-priori approximation rates for the Galerkin POD
	reduced basis which account for
    \begin{itemize}
        \item the truncation in the parameter space,
        \item the Galerkin-error for \cref{eq:motivationvarprob},
        \item and the sampling error in the high-dimensional parameter space.
    \end{itemize}
	\item[(ii)]
	Provide a-priori approximation rates for the \emph{algorithmically implementable}
	and \emph{computationally feasible} Galerkin POD-NN using $\tanh$ NNs up to the training error with dimension-robust convergence rates. Our NN approximation rates account for all arising discretization errors, including
    \begin{itemize}
        \item the truncation in the parameter space,
        \item the Galerkin-error for \cref{eq:motivationvarprob},
        \item the sampling error in the high-dimensional parameter space,
        \item and the neural network approximation error.
    \end{itemize}
	\item[(iii)]
	Demonstrate the validity of our approximation estimates
	on an industrially relevant application: acoustic wave scattering
	by random domains in three spatial dimensions.
\end{enumerate}

In addition, as opposed 
to existing works addressing these issues, 
we employ these theoretically obtained results to guide the NN training in the implementation mentioned in (iii).

\subsection{Outline}
This work is structured as follows. In \cref{sec:projection_ROM} we recall important concepts
concerning the analytic smoothness of the parametric maps upon the parametric inputs. In addition, we recall the projection-based MOR and, in particular, the reduced basis method.
In \cref{sec:analysis_galerkin_POD_RBM}, we provide a complete error analysis for the Galerkin-POD method. In \cref{sec:fully_discrete_GPOD_NN}
we formally introduce NNs and provide a complete error analysis for the approximation of the Galerkin-POD NN.
In \cref{eq:sound_soft_scattering} we introduce the model problem to be considered for the numerical experiments. We conclude 
this work by presenting a set of numerical experiments in \cref{sec:numerical_experiment} and provide some final 
remarks in \cref{sec:concluding_remarks}.

\section{Parametric Problems and projection-based MOR}
\label{sec:projection_ROM}
In this section, we recall important aspects concerning parametric PDEs, parametric holomorphy, and the so-called projection-based \emph{model-order reduction (MOR)}, in particular the reduced basis method. 

\subsection{Encoder-decoder construction}
Firstly, we discuss the construction of both the encoder
and decoder in our approach. For the former, loosely speaking,
we assume that each element of the parameter space $\mathcal{U}$
can be represented through a sequence of real numbers.
More precisely, we set $\mathbb{U} \coloneqq [-1,1]^{\mathbb{N}}$
and assume that for each element of $\nu \in \mathcal{U} \subset \mathcal{X}$
there exists $\y \in \mathbb{U}$ such that $\nu =T(\y)$, where
$T: \mathbb{U} \mapsto \mathcal{U}$. 

Typically, the map is chosen to be affine with respect to the parametric input, i.e.
it is of the form
\begin{equation}\label{eq:map_T}
	T(\y)
	=
	\nu_0
	+
	\sum_{j\geq1}
	y_j
	\nu_j
	\in \mathcal{U},
	\quad
	\y \in \mathbb{U},
	\quad
	\left\{\norm{\nu_j}_\mathcal{X}\right\}_{j\in \mathbb{N}} \in \ell^1(\mathbb{N}),
\end{equation}
where $\{\nu_0,\nu_1,\dots\} \subset \mathcal{X}$,
together with its dimension--truncated counterpart
\begin{equation}\label{eq:map_T_s}
	T_s(\y)
	=
	\nu_0
	+
	\sum_{j=1}^{s}
	y_j
	\nu_j
	\in \mathcal{U},
	\quad
	\y \in \mathbb{U}^{(s)}.
\end{equation}

As discussed in \cite[Section 1.2]{CD15}, for the case of $\mathcal{X}$ a Banach space
and $\mathcal{U}$ a compact subset, it can be proven that maps $T$ of the form described in \cref{eq:map_T} do indeed exist. However, as discussed therein 
this representation might not be unique and not every element of $\mathcal{U}$
might admit one. This issue is, for example, addressed in \cite{herrmann2024neural}
by resorting to frames and Riesz bases in the construction of representation of
the form \cref{eq:map_T}.

Furthermore, in the statistical context and under the assumption that $\mathcal{X}$
is a Hilbert space, an expansion can be constructed as in \cref{eq:map_T}
using the Karhunen--Lo\`{e}ve theorem \cite{Loe1978}, which is indeed the approach that we follow in the application presented in \cref{eq:sound_soft_scattering}.
Herein, we assume that
\begin{equation}
    \mathcal{U} \coloneqq \{\nu = T(\y): \; \y \in \mathbb{U}\},
\end{equation}
and that $\{\nu_0,\nu_1,\dots\}$ form a basis of $\mathcal{U}$, thus rendering
the encoder
\begin{equation}
	\mathscr{E}(\nu) = T^{-1}(\nu)
\end{equation}
well-defined.
For the construction of the decoder, as extensively discussed in
\cref{sec:introduction_related_work} we use a reduced basis $\{\zeta^{\text{(rb)}}_1,\dots, \zeta^{\text{(rb)}}_M\}$ of dimension $M$ constructed using the 
POD approach. Then, the decoder reads as follows
\begin{equation}
    \mathscr{D}(x)
    \coloneqq
    \sum_{i=1}^{M}
    x_i
    \zeta^{\text{(rb)}}_i,
    \quad
    x= \{x_i\}_{i=1}^{M}. 
\end{equation}
The remainder of this section is dedicated to the computational construction of such a reduced basis for parametrically holomorphic maps.

\subsection{Parametric holomorphy}
For $s>1$ we define 
the Bernstein ellipse
\begin{align}
	\mathcal{E}_s
	\coloneqq 
	\left\{
		\frac{z+z^{-1}}{2}: \; 1\leq \snorm{z}\leq s
	\right \} 
	\subset \IC.
\end{align}
This ellipse has foci at $z=\pm 1$ and semi-axes of length 
$a\coloneqq  (s+s^{-1})/2$ and $b \coloneqq  (s-s^{-1})/2$.
In addition, we define the tensorized poly-ellipse
\begin{align}
	\mathcal{E}_{\boldsymbol{\rho}} 
	\coloneqq  
	\bigotimes_{j\geq1} 
	\mathcal{E}_{\rho_j} \subset \IC^{\mathbb{N}},
\end{align}
where $\boldsymbol\rho \coloneqq  \{\rho_j\}_{j\geq1}$
is such that $\rho_j>1$, for $j\in \mathbb{N}$.

\begin{definition}[{\cite[Definition 2.1]{CCS15}}]\label{def:bpe_holomorphy}
Let $X$ be a complex Banach space equipped with the norm $\norm{\cdot}_{X}$. 
For $\boldsymbol{b} \in \ell^p(\mathbb{N})$ with $p\in(0,1)$ and $\varepsilon>0$, we say that map 
$\mathbb{U}  \ni  \y \mapsto  u(\y)  \in  X$
is \emph{$(\boldsymbol{b},p,\varepsilon)$-holomorphic} if and only if:
\begin{enumerate}
	\item\label{def:bpe_holomorphy1}
	The map $\mathbb{U} \ni {\y} \mapsto u(\y) \in X$ is uniformly bounded.
	\item\label{def:bpe_holomorphy2}
	  For any sequence 
	$\boldsymbol\rho\coloneqq \{\rho_j\}_{j\geq1}$ 
	of numbers strictly larger than one that is
	$(\boldsymbol{b},\varepsilon)$-admissible, i.e.~satisfying
	$\sum_{j\geq 1}(\rho_j-1) b_j  \leq  \varepsilon$,
	the map $\y \mapsto u(\y)$ admits a complex
	extension $\z \mapsto u(\z)$ 
	that is holomorphic with respect to each
	variable $z_j$ on a set of the form 
	\begin{align}
		\mathcal{O}_{\boldsymbol\rho} 
		\coloneqq  
		\displaystyle{\bigotimes_{j\geq 1}} \, \mathcal{O}_{\rho_j},
	\end{align}
	where
	\begin{equation} 
	\mathcal{O}_{\rho_j}=
	\{z\in\IC\colon\operatorname{dist}(z,[-1,1])<\rho_j-1\}.
	\end{equation}
	\item
	There exists a constant $C_\varepsilon>0$ such that this extension is bounded on $\mathcal{E}_{\boldsymbol\rho}$ according to
	\begin{equation}
		\sup_{\z\in \mathcal{E}_{\boldsymbol{\rho}}} \norm{u(\z)}_{X}  \leq C_\varepsilon.
	\end{equation}
\end{enumerate}
Without loss of generality, we assume that the sequence $\boldsymbol{b}$ is nonincreasing.
\end{definition}

\subsection{Model Problem: Parametric Variational Problems}\label{sec_rom}
Let $X$ be a \emph{complex} Hilbert space equipped with 
the inner product $\dotp{\cdot}{\cdot}_X$, induced norm
$\norm{\cdot}_X$, and with the associated Banach space of continuous
sesquilinear forms
\begin{equation}
	B(X)
	=
	\left\{
		\mathsf{a}\colon X\times X\to\IC: \;\; \|\mathsf{a}\|_{\text{op}}<\infty
	\right\}
\end{equation}
equipped with the norm
\begin{equation}
	\norm{\mathsf{a}}_{\text{op}}
	\coloneqq
	\sup \limits_{v,w \in X \backslash \{0\}} 
	\frac{|\mathsf{a}(v,w)|}{\|v\|_X \|w\|_X}.
\end{equation}
For each $\y \in \mathbb{U}$, we consider the paramerized variational problem of finding $u(\y) \in X$
such that
\begin{equation}
	\label{eq:continuous_var_prb_a}
	\mathsf{a}(u(\y),v;\y)
	=
	\mathsf{f}(v;\y),
	\quad
	\forall v \in X,
\end{equation}
where $\mathsf{f}(\cdot;\y) \in X'$ and $\mathsf{a}(\cdot,\cdot;\y)\in B(X)$. We assume $\mathsf{f}(\cdot;\y)$ and $\mathsf{a}(\cdot,\cdot;\y)$ to be uniformly continuous, i.e., to satisfy
\begin{equation}\label{eq:lcontinuity}
	\snorm{\mathsf{f}(v;\y)}\leq\gamma\|v\|_X,
	\qquad
	\forall v\in X,\quad \y\in\mathbb{U},
\end{equation}
and
\begin{equation}\label{eq:acontinuity}
	\snorm{\mathsf{a}(u,v;\y)}\leq\overline{\alpha}\|u\|_X\|v\|_X,
	\qquad
	\forall u,v\in X,\quad \y\in\mathbb{U},
\end{equation}
we assume $\mathsf{a}(\cdot,\cdot;\y)$ to be uniformly inf-sup stable, i.e., to satisfy
\begin{equation}\label{eq:ainfsup}
	\inf_{u\in X}\sup_{v\in X}
	\frac{\snorm{\mathsf{a}(u,v;\y)}}{\|u\|_X\|v\|_X}\geq\underline{\alpha},
	\qquad
	\forall u,v\in X,\quad \y\in\mathbb{U},
\end{equation}
for some constants $\gamma\in(0,\infty)$ and $0<\underline{\alpha}\leq\overline{\alpha}<\infty$ which are independent of $\y$, and we assume that for all $v\in X\setminus\{0\}$, $\y\in\mathbb{U}$, there exists $u\in X$ such that $\mathsf{a}(u,v;\y)\neq 0$.
Under these conditions, the Babu\v{s}ka-Aziz theorem implies that for each $\y \in \mathbb{U}$
there exists a bounded solution operator $\mathsf{S}(\cdot,\y) \colon X'\to X$ for each $\y \in \mathbb{U}$ with operator norm uniformly bounded on $\mathbb{U}$.

For the Galerkin discretization of \cref{eq:continuous_var_prb_a} let 
$\{X_h\}_{h>0}\subset X$ be a sequence of one-parameter
finite-dimensional subspaces that are densely embedded in $X$. The discrete variational formulation of \cref{eq:continuous_var_prb_a} is then to find $u_h(\y) \in X_h$
such that
\begin{equation}
\label{eq:discrete_var_prb_a}
    \mathsf{a}(u_h(\y),v_h;\y)
    =
    \mathsf{f}(v_h;\y),
    \quad
    \forall v_h \in X_h.
\end{equation}
We note that $\mathsf{a}(\cdot,\cdot;\y)$ and $\mathsf{f}(\cdot;\y)$ are uniformly continuous on $X_h\subset X$ with the same constants as in \cref{eq:lcontinuity} and \cref{eq:acontinuity} and assume $\mathsf{a}(\cdot,\cdot;\y)$ to be uniformly inf-sup stable on $X_h$, i.e., we assume that it holds
\begin{equation}\label{eq:ainfsupdisc}
	\inf_{w_h\in X_h}\sup_{v_h\in X_h}
	\frac{\mathsf{a}(w_h,v_h;\y)}{\|w_h\|_X\|v_h\|_X}\geq\underline{\alpha},
	\qquad
	\y\in\mathbb{U},
\end{equation}
with (without loss of generality) the same constant as in \cref{eq:acontinuity},
which is independent of the discretization parameter $h$. We further assume that for all $v_h\in X_h\setminus\{0\}$, $\y\in\mathbb{U}$, there exists $u_h\in X_h$ such that $\mathsf{a}(u_h,v_h;\y)\neq 0$.
Again, the Babu\v{s}ka-Azis theorem implies that for each $\y \in \text{U}$ there exists a discrete solution
operator $\mathsf{S}_{h}(\cdot,\y)\colon X'\to X_h$ for each $\y \in \mathbb{U}$ whose operator norms are uniformly bounded on $\mathbb{U}$ by the same constant as for the continuous case.

Moreover, in the following we assume that
\begin{equation}\label{eq:bpeblflf}
\mathsf{a}\colon\mathbb{U}\to B(X),\qquad
\mathsf{f}\colon\mathbb{U}\to X'
\end{equation}
are $(\boldsymbol{b},p,\varepsilon)$-holomorphic and continuous mappings in the sense of \cref{def:bpe_holomorphy}. These assumptions then imply that the \emph{parameter-to-solution map} $\y\mapsto u(\y)$ defined through \cref{eq:continuous_var_prb_a} and the \emph{discrete parameter-to-solution map} $\y\mapsto u_h(\y)$ defined through \cref{eq:discrete_var_prb_a}
are also $(\boldsymbol{b},p,\varepsilon)$-holomorphic and continuous, see, e.g.,~\cite{CCS15}.

\subsection{Proper Orthogonal Decomposition}\label{sec:POD_Exact}
Usually, the numerical approximation of $u_h(\y) \in X_h$
for each instance of the parametric input $\y \in \mathbb{U}$ is computationally demanding, thus rendering any application
that requires a repeated evaluation of the parameter-to-solution map prohibitively expensive. 

Consider the \emph{discrete solution manifold}
\begin{equation}
	\mathcal{M}_h\coloneqq\{u_h(\y)\colon\y\in\mathbb{U}\} \subset X_h.
\end{equation}
We aim to approximate $\mathcal{M}_h$ by low-dimensional linear subspaces following the reduced basis method, see, e.g., \cite{OR2016,QMNMOR2016} and the references therein.  
More precisely, we seek a $J$-dimensional subspace $X_{h,J}\subset X_h$
minimizing the projection error in the $L^2(\mathbb{U};X)$ sense,~i.e.
\begin{equation}\label{eq:rom_basis_cont}
	X_{h,J}^{(\text{rb})}
	=
	\argmin_{
	\substack{X_{h,J} \subset X_h\\\dim X_{h,J}\leq J}}
	\varepsilon_{h}
	(X_{h,J}),
\end{equation}
where
\begin{equation}\label{eq:error_measure_continuous}
\begin{aligned}
	\varepsilon_h
	\left(
        X_{h,J}
	\right)
    &
	\coloneqq
	\|u_h- \mathsf{P}_{X_{h,J}}u_h\|_{L^2(\mathbb{U};X)}^2
    \\
    &
	=
	\int\limits_{\y \in \mathbb{U}}
	\norm{
		u_h(\y)
		-
		\mathsf{P}_{X_{h,J}}
		u_h(\y)
	}^2_X
	\mu(\text{d}\y).
\end{aligned}
\end{equation}
Therein, we define the operator $\mathsf{P}_{X_{h,J}}: X \rightarrow X_{h,J}$ 
as the orthogonal projection operator onto $X_{h,J}$ with respect to $(\cdot,\cdot)_X$
and introduce the following tensor product uniform measure in $\mathbb{U}$
\begin{equation}
    \mu(\text{d}\y)
    =
    \bigotimes_{j\in \mathbb{N}} \frac{\text{d}y_j}{2}.
\end{equation}
We also note that \cref{def:bpe_holomorphy1} in the definition of the $(\boldsymbol{b},p,\varepsilon)$-holomorphy implies that $u_h\in L^2(\mathbb{U};X_h)$.
Therefore, the operators
\begin{align}
	\mathsf{T}_h\colon& L^2(\mathbb{U})\to X_h,&
	g\mapsto
	\mathsf{T}_h g 
	&=
	\int_{\mathbb{U}}
	u_h(\y) g(\y) 
	\mu(\dd\y),\\
	\mathsf{T}^\star_h\colon& X_h\to L^2(\mathbb{U}),&
	x\mapsto
	\mathsf{T}^\star_h x
	&=
	\big(u_h(\y),x\big)_X
\end{align}
are Hilbert-Schmidt ones, and thus compact. This makes the integral operator
\begin{align}\label{eq:continuousio}
	\mathsf{K}_h\colon X_h\to X_h,\qquad
	x\mapsto
	\mathsf{K}_hx
	=
	\mathsf{T}_h\mathsf{T}^\star_h x
	=
	\int_{\mathbb{U}}u_h(\y)\big(u_h(\y),x\big)_X \mu(\dd\y),
\end{align}
compact, self-adjoint, and positive definite. 
Consequently, it has a countable sequence of eigenpairs
$(\zeta_{h,i},\sigma_{h,i}^2)_{i=1}^r \in X_h\times\mathbb{R}_{\geq0}$,
being $r \in \mathbb{N}$ the rank of the operator $\mathsf{T}_h$, with the eigenvalues accumulating at zero. In the following, we assume that $\sigma_{h,1}\geq \sigma_{h,2}\geq \cdots\geq \sigma_{h,r}\geq 0$.
Moreover, it is well-known that the span of the eigenfunctions to the $J$ largest eigenvalues, referred to as \emph{reduced basis},
\begin{align}\label{eq:rb_exact}
	X_{h,J}^{(\text{rb})}
	=
	\operatorname{span}
    \left\{
        \zeta_{h,1}
        ,\ldots,
        \zeta_{h,J}
    \right\}
	\subset X_h,
\end{align}
minimizes the projection error \cref{eq:error_measure_continuous}, that is
in the $L^2(\mathbb{U};X_h)$ sense,
among all $J$-dimensional subspaces of $X_h$ 
of dimension at most $J$ to
\begin{align}\label{eq:bounderrmeas}
    \varepsilon_h
    \left(
        X_{h,J}^{(\text{rb})}
    \right)
    =
    \sum_{i=J+1}^r\sigma_{h,i}^2.
\end{align}

Recall that for a compact subset $\mathcal{K}$ of a Banach space
$X$ the Kolmogorov's width is defined for $J \in \IN$ as  
\begin{align}
	d_J(\mathcal{K},X)
	\coloneqq
	\inf _{\substack{X_J\subset X\\\operatorname{dim}\left(X_J\right) \leq J}} 
	\sup _{v \in \mathcal{K}} 
	\inf_{w \in X_J}
	\norm{v-w}_X.
\end{align}

Considering that in our case $X$ is a Hilbert space, one can readily observe that
$\sqrt{\varepsilon_h \left( X_{h,J}^{(\text{rb})} \right)} \leq d_J(\mathcal{M}_h,X_h)$. 
In \cite[Theorem 2.1]{DLS2016}, the decay of the
Kolmogorov's width under the application of holomorphic maps has been studied and convergence rates are provided. 
In this work, we restrict ourselves to a slightly more specific setting. That is, we work under the assumption that the parameter-to-solution map is $(\boldsymbol{b},p,\varepsilon)$-holomorphic. As a consequence of this property, by performing a multivariate 
polynomials expansion as in \cite[Corollary 5.2]{dolz2024parametric}, one can show that
\begin{equation}\label{eq:rbbestNterm}
	\varepsilon_h
    \left(
        X_{h,J}^{(\text{rb})}
    \right)
	\lesssim
	J^{-2\left(\frac{1}{p}-1\right)},
\end{equation}
with an implicit constant depending only on $p\in (0,1)$ and $\boldsymbol{b}$, however not on $h$ or $J$.

\subsection{Empirical POD}
\label{sec:galerkin_POD}
The construction of the reduced basis $X_{h,J}^{(\text{rb})}$
introduced in \cref{eq:rb_exact} as described in \cref{sec:POD_Exact}
is not feasible in practice as $\mathsf{K}_h$ is not computationally accessible.
To overcome this issue, one seeks a $J$-dimensional subspace
\begin{equation}
	X_{h,s,N,J}^{(\text{rb})}
	=
	\argmin_{
	\substack{X_{h,J} \subset X_h\\\dim X_{h,J}\leq J}}
	\varepsilon_{h,s,N}
	(X_{h,J}),
\end{equation}
which is the unique minimizer of the \emph{computable or empirical} error measure
\begin{align}\label{eq:discrete_error_measure}
	\varepsilon_{h,s,N}
	\left(
		X_{J}
	\right)
	\coloneqq
	\frac{1}{N}
	\sum_{n=1}^{N}
	\norm{
		u_h
		\left(
			\y^{(n)}
		\right)
		-
		\mathsf{P}_{X_J}
		u_h
		\left(
			\y^{(n)}
		\right)
	}^2_X,
\end{align}
with sample points
$\left\{\y^{(n)}\right\}_{n=1}^N \subset \mathbb{U}^{(s)}\coloneqq [-1,1]^s$, $s\in\mathbb{N}$. 
Similarly to $\mathbb{U}$, we equip $\mathbb{U}^{(s)}$ with the structure of a probability space
and with the tensor product unit measure
\begin{align}
	\mu^s(\dd\y)
	=
	\bigotimes_{j=1}^{s} \frac{\dd y_j}{2}.
\end{align}
We assume that these sample points are taken to be
quasi-Monte Carlo points such as the Halton point sequences \cite{Hal1960}
or higher order quasi-Monte Carlo based on IPL sequences, see
e.g. \cite{DGGS16,DKL14,DLS2016}.

We further assume to have a basis $\{\varphi_{1},\dots,\varphi_{N_h}\}$ of $X_h$ at our disposal and denote by boldface letters $\mathbf{w}_h\in\mathbb{C}^{N_h}$ the coefficient vector of a function $w_h\in X_h$ in the aforementioned basis.
Using the mass matrix
${\bf M}_h \in \mathbb{R}^{N_h\times N_h}$ defined as
\begin{align}
	\left(
		{\bf M}_h
	\right)_{i,j}
	=
	\dotp{\varphi_{i}}{\varphi_{j}}_X,
	\quad
	i,j
	\in \{1,\dots,N_h\},
\end{align}
this one-to-one correspondence yields finite dimensional representations
of norm and inner product in $X_h$, which read
\begin{equation}
	\dotp{v_h}{w_h}_X
	=
	{\bf v}_h^{\star}{\bf M}_h{\bf w}_h
	\quad
	\text{and}
	\quad
	\norm{v_h}_X
	=
	\sqrt{
		{\bf v}_h^{\star}{\bf M}_h{\bf v}_h
	}
    \defis
    \norm{{\bf v}_h}_{{\bf M}_h},
\end{equation}
for $v_h,w_h \in X_h$. We recall that ${\bf M}_h$ is symmetric and positive definite.

Let $X_{h,J}$ be a subspace of $X_h$ of dimension $J$ which is spanned
by the orthonormal basis $\{v_h^{(1)},\dots,v_h^{(J)}\}$.
Set ${\bf\Phi}=\left({\bf v}_h^{(1)},\dots,{\bf v}_h^{(J)}\right)$
where each ${\bf v}_h^{(i)} \in \mathbb{C}^{N_h}$ collects the 
coefficients of the representation of $v^{(i)}_h$ in the basis
$\{\varphi_{1},\dots,\varphi_{N_h}\}$ of $X_h$.
Using these facts, \cref{eq:discrete_error_measure}
becomes %
\begin{equation}
\begin{aligned}
	\varepsilon_{h,s,N}
	\left(
		X_{h,J} %
	\right)
	&
	=
	\frac{1}{N}
	\sum_{n=1}^{N}
	\norm{
	{\bf u}_h
	\big(
		\y^{(n)}
	\big)
	-
	\sum_{j=1}^J
	\Big(\big({\bf v}^{(j)}_h\big)^\star{\bf M}_h{\bf u}_h\big(\y^{(n)}\big)\Big)
	{\bf v}^{(j)}_h
	}^2_{{\bf M}_h}
	\\
	&
	=
	\frac{1}{N}
	\sum_{n=1}^{N}
	\norm{
	{\bf u}_h
	\big(
		\y^{(n)}
	\big)
	-
	{\bf \Phi}
	{\bf \Phi}^\star
	{\bf M}_h
	{\bf u}_h
	\big(
		\y^{(n)}
	\big)
	}^2_{{\bf M}_h}.
\end{aligned}
\end{equation}

To compute the minimum of this error measure, we define the \emph{snapshot matrix} $\widetilde{\mathbf{S}}$
as
\begin{align}\label{eq:snapshot_matrix}
  \widetilde{\mathbf{S}}
	\coloneqq
	\left(
		{\bf u}_h
		\big(
			\y^{(1)}
		\big)
		,
		\dots
		,
		{\bf u}_h
		\big(
			\y^{(N)}
		\big)
	\right)
	\in
	\IC^{N_h \times N},
\end{align}
where, as previously explained, each ${\bf u}_h\left(\y^{(i)}\right)$
corresponds to the representation in the basis of $X_{h,J}$ 
of $u_h\left(\y^{(i)}\right)$.
Considering the SVD $\mathbf{S} = \mathbf{U}\boldsymbol{\Sigma}\mathbf{V}^{\dagger}$ of $\mathbf{S} = N^{-1/2}{\bf M}_h^{1/2}\widetilde{\mathbf{S}}$, where
\begin{equation}
	{\bf U}
	=
	\left(
		{\boldsymbol{\zeta}}_1,
		\ldots,
		{\boldsymbol{\zeta}}_{N_h}
	\right)\in \mathbb{R}^{N_h \times N_h},
	\quad 
	{\bf V}
	=
	\left(
		{\boldsymbol{\psi}}_1,
		\ldots,
		{\boldsymbol{\psi}}_{N_h}
	\right) 
	\in 
	\mathbb{R}^{N \times N},
\end{equation}
are orthogonal matrices and 
$\boldsymbol{\Sigma}=\operatorname{diag}
\left(\sigma_{h,s,N,1}, \ldots, \sigma_{h,s,N,\check{r}}\right) 
\in \mathbb{R}^{N_h \times N}$ with $\sigma_{h,s,N,1} \geq \cdots \geq \sigma_{h,s,N,\check{r}}>0$, 
being $\check{r} \in \mathbb{N}$ the rank of 
$\widetilde{\mathbf{S}}$, we obtain through POD the following basis of reduced dimension $J$
\begin{equation}\label{eq:reduced_space}
	{\bf \Phi}^{\text{(rb)}}_J
	=
	\left(
		{\boldsymbol\zeta}^{\text{(rb)}}_1,
		\dots,
		{\boldsymbol\zeta}^{\text{(rb)}}_J
	\right)
	=
	\left(
	{\bf M}_h^{-1/2}
	{\boldsymbol\zeta}_1,
	\dots,
	{\bf M}_h^{-1/2}
	{\boldsymbol\zeta}_J
	\right)
\end{equation}
for $J\leq \check{r}$. 
This basis is such that its span 
understood as elements of $X_h$, which in 
the following we refer to as $X_{N,s,h,J}^{(\text{rb})}$, satisfies
\begin{equation}\label{eq:PODsingularvalues}
	\varepsilon_{h,s,N}
	\Big(X_{h,s,N,J}^{(\text{rb})}\Big)
	=
	\min_{
		\substack{X_{h,J} \subset X_h\\\dim X_{h,J}\leq J}}
	\varepsilon_{h,s,N}
	\Big(X_{h,J}\Big)
	=
	\sum_{i=J+1}^{\check{r}}
	\sigma_{h,s,N,i}^2,
\end{equation}
see, e.g.,~\cite[Proposition 6.2]{QMNMOR2016}.

\begin{remark}\label{rem:PODinnerproduct}
  Rather than applying ${\bf M}_h^{\pm1/2}$,
  in actual computations one would evaluate
  \begin{equation}\label{eq:snapshotcorrelation}
    \mathbf{C}=\frac{1}{N}\widetilde{\mathbf{S}}^\star\mathbf{M}_h\widetilde{\mathbf{S}},
  \end{equation}
  exploit that
$
    \mathbf{C}\boldsymbol{\psi}_i
    =	\mathbf{S}^\star\mathbf{S}\boldsymbol{\psi}_i
    =
    \sigma_{h,s,N,i}^2\boldsymbol{\psi}_i,
$
compute the eigenpairs corresponding to the $J$ largest eigenvalues of $\mathbf{C}$, and set $\boldsymbol{\zeta}_i=\sigma_i^{-1}\mathbf{S}\boldsymbol{\psi}_i$, $i=1,\ldots,J$, see also \cite{QMNMOR2016}.
\end{remark}

\section{Fully Discrete Analysis of the Galerkin-POD RB Method}
\label{sec:analysis_galerkin_POD_RBM}
In this section, we provide a complete error analysis of the Galerkin-POD RB method.
\subsection{Galerkin POD Error Estimate}
The goal of this section is to bound the error between the
solution $u$ to \cref{eq:continuous_var_prb_a} and $\mathsf{P}_{X_{h,s,N,J}^{(\text{rb})}}u_h$
with ${X}_{h,s,N,J}^{(\text{rb})}$ as in \cref{eq:PODsingularvalues}
and $u_h$ as in \cref{eq:discrete_var_prb_a}
in terms of the following error sources and corresponding
discretization variables: (i) Galerkin discretization ($h>0$), (ii)
dimension truncation of the parametric input ($s \in \mathbb{N}$),
(iii) reduced basis approximation ($J \in \mathbb{N}$), and
(iv) number of snapshots used in the empirical computation of the reduced basis ($N \in \mathbb{N}$). 
To this end, we note that the error itself can be split into the following contributions:
\begin{equation}\label{eq:errordecomposition}
\begin{aligned}
\norm{u- \mathsf{P}_{X_{h,s,N,J}^{(\text{rb})}}
    u_h
}_{L^2(\mathbb{U};X)}
\lesssim
\underbrace{
\norm{
    u
    -
    u_h
}_{L^2(\mathbb{U};X)}
}_{\text{Galerkin error}}
+
\underbrace{
\norm{
    u_h
    -
    u^{(s)}_h 
}_{L^2(\mathbb{U};X)}
}
_{\substack{\text{Truncation}\\\text{Error}}}&\\
+
\underbrace{
\norm{
    u^{(s)}_h
    -
    \mathsf{P}_{X_{h,s,NJ}^{(\text{rb})}}
    u^{(s)}_h 
}_{L^2(\mathbb{U};X)}
}_{\substack{\text{POD Error}}}&
,
\end{aligned}
\end{equation}
where for $\y = (y_1,\dots,y_s,\dots) \in \mathbb{U}$  we set 
$u^{(s)}_h(\y) = u_h(y_1,y_2,\dots,y_s,0,0\dots)$.
The implicit constant in \cref{eq:errordecomposition}
is independent of $h$, $s$, $N$, and $J$.

To estimate the Galerkin error, we note that standard
inf-sup theory yields the following estimate, which is valid pointwise 
for each $\y \in \mathbb{U}$
\begin{equation}
\|u(\y)-u_h(\y)\|_X
\leq
\bigg(1+\frac{\overline{\alpha}}{\underline{\alpha}}\bigg)
\inf_{v_h\in X_h}\|u(\y)-v_h\|_X
\end{equation}
and by exploiting that $\mathbb{U}$ has unit measure, we obtain
\begin{equation}\label{eq:Galerkinerror}
\begin{aligned}
\norm{
    u
    -
    u_h
}_{L^2(\mathbb{U};X)}
&
\leq
\norm{
    u
    -
    u_h
}_{L^\infty(\mathbb{U};X)}
\\
&
\leq
\bigg(1+\frac{\overline{\alpha}}{\underline{\alpha}}\bigg)
\sup_{\y\in\mathbb{U}}
\inf_{v_h\in X_h}\|u(\y)-v_h\|_X.
\end{aligned}
\end{equation}
To estimate the truncation error, we exploit again that $\mathbb{U}$ has unit measure and that the $(\boldsymbol{b},p,\varepsilon)$-holomorphy of $u_h$ yields%
\begin{equation}\label{eq:truncationerror}
    \norm{
        u_h
        -
	   u^{(s)}_h
    }_{L^\infty(\mathbb{U};X)}
    \lesssim
    s^{-\left(\frac{1}{p}-1\right)},
\end{equation}
with the constant depending only on $p\in (0,1)$ and $\boldsymbol{b} \in \ell^p(\mathbb{N})$.
We proceed to prove this claim. Observe that
\begin{equation}\label{eq:triangle}
    \norm{
        u_h
        -
        u^{(s)}_h
    }_{L^\infty(\mathbb{U};X)}
    \leq
    \sum_{k=s+1}^\infty
    \norm{
        u^{(k+1)}_h
        -
        u^{(k)}_h
    }_{L^\infty(\mathbb{U};X)}
\end{equation}
where we have used that
$u_h = \lim_{k\rightarrow \infty} u^{(k)}_h $.
Next, observe that
\begin{equation}\label{eq:bound_der}
    \norm{
        u^{(k+1)}_h
        -
        u^{(k)}_h
    }_{L^\infty(\mathbb{U};X)}
    \leq
    2
    \sup_{\y  \in \mathbb{U}}
    \norm{
        \left(
            \partial_{k+1}
            u_h
        \right)(\y)
    }_X,
\end{equation}
where $\partial_{k+1}$ denotes partial differentiation
with respect to the $k+1$ component of the parametric 
input $\y \in \mathbb{U}$. In follows from 
\cite[Theorem 3.1]{DLS2016} that there exists a finite
$K \in \mathbb{N}$ such that for any $k\in \mathbb{N}$
one has
\begin{equation}\label{eq:bound_der_2}
    \sup_{\y  \in \mathbb{U}}
    \norm{
        \left(
            \partial_{k+1}
            u_h
        \right)(\y)
    }_X
    \lesssim
    \beta_{k} \coloneqq 
    \left\{
    \begin{array}{cc}
        1 & k<K, \\
        b_k &  k>K.
    \end{array}
    \right.
\end{equation}
Clearly, one has that $\boldsymbol{\beta} \coloneqq \{\beta_j\}_{j\geq 1} \in \ell^p(\mathbb{N})$, with the same $p \in (0,1)$, therefore
it follows from \cref{eq:triangle,eq:bound_der,eq:bound_der_2} and \cite[Theorem 2.1]{DLS2016} that \cref{eq:truncationerror}
holds true.
It thus remains to bound the POD error.

\subsection{POD Sampling Error}
To estimate the POD error, we note that 
\begin{equation}\label{eq:continuous_risk}
\norm{
    u_h^{(s)}
    -
    \mathsf{P}_{X_{h,s,N,J}^{(\text{rb})}}u_h^{(s)}
}^2_{L^2(\mathbb{U};X)}
=
\int_{\mathbb{U}^{(s)}}
\norm{
    u_h^{(s)}(\y)
    -
    \mathsf{P}_{X_{h,s,N,J}^{(\text{rb})}}u_h^{(s)}(\y)
    }^2_{X}
    \mu^s(\dd\y),
\end{equation}
and that \cref{eq:discrete_error_measure} is obtained by 
applying an equal weights, $N$-points quadrature rule with 
points $\left\{\y^{(n)}\right\}_{n=1}^N \subset \mathbb{U}^{(s)}$, $s\in\mathbb{N}$ to \cref{eq:continuous_risk}.
As we show in the following, 
quasi-Monte Carlo and higher-order quasi-Monte Carlo estimates, such as
the Halton sequence or the IPL sequences, see \cref{sec:QMC} for details,
are now immediately applicable.

\begin{lemma}\label{lem:QMCerror}
It holds 
\begin{align}\label{eq:PODerror}
\snorm{
\norm{
    u_h^{(s)}
    -
    \mathsf{P}_{X_{h,s,N,J}^{(\normalfont\text{rb})}}
    u_h^{(s)}
    }^2_{L^2(\mathbb{U}^{(s)};X)}
	-
	\varepsilon_{h,s,N}
	\left(
		X_{h,s,N,J}^{(\normalfont\text{rb})}
	\right)
}
	\lesssim
	N^{-\alpha}.
\end{align}
Here, we obtain $\alpha=1-\delta$ for any $\delta\in(0,1)$ for the Halton sequence 
under the assumption $p \in (0,\frac{1}{3})$ 
and $\alpha=\frac{1}{p}$ for the IPL sequences. In the former case, the implicit constant in \cref{eq:PODerror} depends on $\delta$, 
and tends to infinity as $\delta\rightarrow 0^+$.
\end{lemma}

\begin{proof}
Under the assumptions established in \ref{sec_rom}, the map 
$\mathbb{U} \ni \y \mapsto u_h(\y) \in X$ is  $(\boldsymbol{b},p,\varepsilon)$-holomorphic
and continuous. Next, it follows from \Cref{lmm:bpe_holomorphy_pod} that
the map
\begin{equation}\label{eq:integrand_hol}
    \mathbb{U}
    \ni
    \y \mapsto 
    \norm{
        u_h(\y)
        -
        \mathsf{P}_{X_{h,s,N,J}^{(\text{rb})}}u_h(\y)
    }^2_{X}
    \in
    \mathbb{R}.
\end{equation}
is so as well. 
Using this, the result of this lemma is a direct consequence
of \Cref{lemma:Halton,prop:QMC_error} in \Cref{sec:QMC} for the Halton and HoQMC
quadrature rules, respectively.
\end{proof}

Using \cref{eq:Galerkinerror,eq:truncationerror,eq:PODerror} to bound the errors in \cref{eq:errordecomposition},
we obtain the following error bound.
\begin{corollary}\label{cor:contest}
It holds
\begin{equation}\label{eq:error_bound_eigenvalues}
\begin{aligned}
    \norm{u- \mathsf{P}_{X_{h,s,N,J}^{(\normalfont\text{rb})}}
        u_h
    }^2_{L^2(\mathbb{U};X)}
    \lesssim
    {}&
	\sup_{\y \in \mathbb{U}}
	\inf_{v_h \in X_h}
	\norm{u(\y)-v_h}_X^2
    \\
    &
    +
    s^{-2\left(\frac{1}{p}-1 \right)}
    +
    N^{-\alpha}
    +
    \varepsilon_{h,s,N}
\left(		X_{h,s,N,J}^{(\normalfont\text{rb})}
	\right),
\end{aligned}
\end{equation}
with $\alpha=1-\delta$ in the case of the Halton sequence under the assumption $p \in (0,\frac{1}{3})$
and $\alpha=\frac{1}{p}$ in the case of the IPL sequences for $p \in (0,1)$. 
In the former case, the hidden constant in
\cref{eq:error_bound_eigenvalues} tends to infinity 
as $\delta \rightarrow 0^{+}$.
\end{corollary}

We note that $\varepsilon_{h,s,N}\big(X_{h,s,N,J}^{(\normalfont\text{rb})}\big)$
can be fully controlled a-posteriori by selecting an appropriate dimension $J$ for the reduced space in \cref{eq:PODsingularvalues}.
In the following, we give an a-priori analysis of this term.

\subsection{A-priori analysis of the POD error}
Considering
\begin{equation}\label{eq:errmeasapriori}
	\varepsilon_{h,s,N}
	\Big(X_{h,s,N,J}^{(\text{rb})}\Big)
	=
	\sum_{i=J+1}^r
	\sigma^2_{h,s,N,i}
\end{equation}
as given in \cref{eq:PODsingularvalues} we observe that it is fully determined by the eigenvalues $\sigma_{h,s,N,i}^2$ of the matrix
$
{\bf C}
=
\frac{1}{N}
\widetilde{\bf S}^\star{\bf M}_h\widetilde{\bf S},
$
see also \cref{rem:PODinnerproduct}.
In the following, we bound $\varepsilon_{h,s,N}\big(X_{h,s,N,J}^{(\text{rb})}\big)$
in terms of $N$ and $J$.

\begin{lemma}\label{lem:discerrdecay}
It holds
\begin{equation}\label{eq:error_bound_quad}
    \varepsilon_{h,s,N}
    \Big(X_{h,s,N,J}^{(\normalfont\text{rb})}\Big)
    \lesssim
    N^{-\alpha}
    +
    J^{-2\left(\frac{1}{p}-1\right)}
\end{equation}
with the same considerations as stated in \Cref{cor:contest} for $\alpha$ and the hidden constant in \cref{eq:error_bound_quad}.
\end{lemma}
\begin{proof}
We observe that $\mathbf{C}_h=\frac{1}{N}\widetilde{\mathbf{S}}^\star\mathbf{M}\widetilde{\mathbf{S}}$ in \cref{eq:snapshotcorrelation}, $\frac{1}{N}({\bf M}_h^\star)^{1/2}\widetilde{\bf S}^\star\widetilde{\bf S}{\bf M}_h^{1/2}$, and ${\bf K}=\frac{1}{N}\widetilde{\bf S}^\star\widetilde{\bf S}{\bf M}_h$ have the very same eigenvalues $\sigma_{h,s,N,i}^2$ and that the latter, ${\bf K}$, is the matrix representation of
\begin{equation}
	\mathsf{K}_{h,s,N}w_h
	=
	\frac{1}{N}
	\sum_{n=1}^Nu_h(\y^{(n)})\Big(u_h(\y^{(n)}),w_h\Big)_X,
\end{equation}
which also has eigenvalues $\sigma_{h,s,N,i}^2$. Using this notation, and recalling \cref{eq:continuousio}, we estimate
\begin{equation}
\begin{aligned}
    \varepsilon_{h,s,N}
        \Big(X_{h,s,N,J}^{(\text{rb})}\Big)
        &=
	\sum_{i=J+1}^{\check{r}}\sigma_{h,s,N,i}^2\\
        &=
	\min_{\substack{\mathbf{v}\in\mathbb{C}^{N_h\times r-J}\\
	\mathbf{v}^\star\mathbf{v}=\mathbf{I}}}\trace\big(\mathbf{v}^\star\mathbf{C}\mathbf{v}\big)\\
      	&=
      	\min_{\substack{V\subset X_h\\\dim V\leq r-J}}\trace\big(\mathsf{P}_V\mathsf{K}_{h,s,N}\mathsf{P}_V\big)\\
      	&=
      	\min_{\substack{V\subset X_h\\\dim V\leq r-J}}
	\left(
		\trace\big(\mathsf{P}_V\mathsf{K}_{h,s,N}\mathsf{P}_V-\mathsf{P}_V\mathsf{K}_h\mathsf{P}_V\big)
    \right.
	\\
    &
    \qquad
    \qquad
    +
    \left.
		\trace\big(\mathsf{P}_V\mathsf{K}_h\mathsf{P}_V\big)
	\right)
\end{aligned}
\end{equation}
  where, for an arbitrary orthonormal basis $\{\chi_i\}_{i=1}^{N_h}$ of $X_h$, it holds
  \begin{equation}\label{eq:comp_trace}
    \begin{aligned}
      \trace\big(&\mathsf{P}_V\mathsf{K}_{h,s,N}\mathsf{P}_V-\mathsf{P}_V\mathsf{K}_h\mathsf{P}_V\big)\\
      &=
      \sum_{i=1}^{N_h}\Big(\big(\mathsf{K}_{h,s,N}-\mathsf{K}_h\big)\mathsf{P}_V\chi_i,\mathsf{P}_V\chi_i\Big)_X\\
      &=
      \sum_{i=1}^{N_h}\Big(\mathsf{K}_{h,s,N}\mathsf{P}_V\chi_i,\mathsf{P}_V\chi_i\Big)_X
      -
      \sum_{i=1}^{N_h}\Big(\mathsf{K}_h\mathsf{P}_V\chi_i,\mathsf{P}_V\chi_i\Big)_X\\
      &=
      \frac{1}{N}\sum_{n=1}^N\sum_{i=1}^{N_h}\big(\mathsf{P}_Vu_h(\y^{(n)}),\chi_i\big)_X^2
      -
      \int_{\mathbb{U}^{(s)}}\sum_{i=1}^{N_h}\big(\mathsf{P}_Vu_h(\y),\chi_i\big)_X^2\dd\mu^{(s)}(\y)\\
      &=
      \frac{1}{N}\sum_{n=1}^N\big\|\mathsf{P}_Vu_h(\y^{(n)})\big\|_X^2
      -
      \int_{\mathbb{U}^{(s)}}\big\|\mathsf{P}_Vu_h(\y)\big\|_X^2\dd\mu^{(s)}(\y)
    \end{aligned}
  \end{equation}
  Observe that the map $\y \mapsto \mathsf{P}_Vu_h(\y)$ is straightforwardly 
  $(\boldsymbol{b},p,\varepsilon)$-holomorphic
  as the application of $\mathsf{P}_V$ is a linear operation, and 
  as a consequence of \Cref{lmm:bpe_holomorphy_pod}
  so is the map 
  \begin{equation}
      \mathbb{U}
      \ni
      \y
      \mapsto
      \big\|\mathsf{P}_Vu_h(\y)\big\|_X^2
      \in
      \mathbb{R}.
  \end{equation}
    It follows from \Cref{lemma:Halton,prop:QMC_error} in \Cref{sec:QMC} for the Halton and HoQMC, respectively, 
    that the last equation in \cref{eq:comp_trace} is bounded 
    in absolute value by
\begin{equation}
    \snorm{
    \frac{1}{N}\sum_{n=1}^N\big\|\mathsf{P}_Vu_h(\y^{(n)})\big\|_X^2
      -
      \int_{\mathbb{U}^{(s)}}\big\|\mathsf{P}_Vu_h(\y)\big\|_X^2\dd\mu^{(s)}(\y)
    }
    \lesssim
    N^{-\alpha},%
\end{equation}
with the considerations for the implicit constant and
$\alpha>0$ indicated in \cref{lemma:Halton,prop:QMC_error}.
  This implies
  \begin{equation}\label{eq:final_bound}
      \varepsilon_{h,s,N}
      \Big(X_{h,s,N,J}^{(\text{rb})}\Big)
      \lesssim 
      N^{-\alpha}+\underbrace{\min_{\substack{V\subset X_h\\\dim V\leq r-J}}\trace\big(\mathsf{K}_h|_V\big)}_{=\varepsilon_h(X_{h,J}^{(\text{rb})})},
  \end{equation}
  where $X_{h,s,N,J}^{(\text{rb})}$ is as in \cref{eq:rb_exact}.
  The last term in \cref{eq:final_bound} can be estimated using \cref{eq:rbbestNterm}, implying the assertion.
\end{proof}

\begin{corollary}\label{cor:full_error_projection}
It holds
\begin{equation}\label{eq:error_bound_quad_2}
\begin{aligned}
    \norm{u- \mathsf{P}_{X_{h,s,N,J}^{(\normalfont\text{rb})}}
        u_h
    }_{L^2(\mathbb{U};X)}
    \lesssim
    {}&
    \sup_{\y \in \mathbb{U}}
    \inf_{v_h \in X_h}
    \norm{u(\y)-v_h}_X
    \\
    &
    +
    s^{-\left(\frac{1}{p}-1\right)}
    +
    N^{-\frac{\alpha}{2}}
    +
    J^{-\left(\frac{1}{p}-1\right)}
\end{aligned}
\end{equation}
with the same considerations stated in \Cref{cor:contest} for $\alpha$ and for the hidden constant in \cref{eq:error_bound_quad_2}.
\end{corollary}
\begin{proof}
Combine \cref{cor:contest} and \cref{lem:discerrdecay}.
\end{proof}
Balancing  errors yields the following a-priori estimate for the ranks of the POD generated subspace $X_{h,s,N,J}^{(\normalfont\text{rb})}\subset X_h$.
\begin{corollary}\label{cor:PODrankbound}
Choosing $J$ in \cref{eq:errmeasapriori} such that 
\begin{equation}
\varepsilon_{h,s,N}
\Big(X_{h,s,N,J}^{(\normalfont\text{rb})}\Big)
\lesssim
N^{-\alpha}
\end{equation}
yields a reduced space $X_{h,s,N,J}^{(\normalfont\text{rb})}\subset X_h$ with dimension at most $J\sim N^{\frac{\alpha}{2(1/p-1)}}$ and satisfying
\begin{equation}
\begin{aligned}
\norm{u- \mathsf{P}_{X_{h,s,N,J}^{(\normalfont\text{rb})}}
    u_h
}_{L^2(\mathbb{U};X)}
\lesssim
\sup_{\y \in \mathbb{U}}
\inf_{v_h \in X_h}
\norm{u(\y)-v_h}_X
+
s^{-\left(1/p-1\right)}
+
N^{-\alpha/2}.
\end{aligned}
\end{equation}
\end{corollary}

\section{Fully Discrete Error Analysis of the Galerkin-POD NN}
\label{sec:fully_discrete_GPOD_NN}
In this section, we discuss the approximation properties of the Galerkin-POD NN.
We are interested in a fully discrete error analysis for the approximation of the parameter-to-solution map by means of NNs by taking into account all the previously discussed error sources.

\subsection{Artificial Neural Networks}
\label{sec:deep_neural_networks}
Let $L \in \IN$, $\ell_0,\dots, \ell_L \in \IN$
and let
$\sigma:\IR \rightarrow\IR$ be a
nonlinear function, referred to in the following as
the \emph{activation function}. Set
\begin{equation}
    \Theta
    \coloneqq
    \bigtimes_{k=1}^{L}
    \Big(\mathbb{R}^{\ell_{k}\times \ell_{k-1}}
    \times
    \mathbb{R}^{\ell_k}\Big).
\end{equation}
For $\boldsymbol \theta = (\boldsymbol \theta_1,\dots,\boldsymbol \theta_L) \in \Theta$,
with $ \boldsymbol \theta_k = ({\bf W}_k,{\bf b}_k)$,
${\bf W}_k\in\IR^{\ell_k\times\ell_{k-1}}$, ${\bf b}_k\in\IR^{\ell _k}$,
consider the affine transformation
${\bf A}_k: \IR^{\ell_{k-1}}\rightarrow \IR^{\ell_k}: {\bf x} \mapsto {\bf W}_k {\bf x}+{\bf b}_k$
for $k\in\{1,\ldots,L\}$. We define a \emph{neural network (NN)}
with activation function $\sigma$ as the map
$\Psi_{\boldsymbol \theta}:\IR^{\ell_0}\rightarrow \IR^{\ell_L}$ defined as 
\begin{align}\label{eq:ann_def}
	\Psi_{\boldsymbol \theta}({\bf x})
	\coloneqq
	\begin{cases}
	{\bf A}_1({\bf x}), & L=1, \\
	\left(
		{\bf A}_L
		\circ
		\sigma
		\circ
		{\bf A}_{L-1}
		\circ
		\sigma
		\cdots
		\circ
		\sigma
		\circ
		{\bf A}_1
	\right)({\bf x}),
	& L\geq2,
	\end{cases}
\end{align}
where the activation function $\sigma:\IR\rightarrow \IR$
is applied componentwise. We define the depth and the width
of an NN as
\begin{equation}
    \normalfont\text{width}(\Psi_{\boldsymbol \theta})
    =\max\{\ell_0,\ldots, \ell_L\}
    \quad
    \text{and}
    \quad
    \normalfont\text{depth}(\Psi_{\boldsymbol \theta})
    =
    L+1,
\end{equation}
respectively. 

In the present work, we consider as activation function
the hyperbolic tangent
\begin{align}
	\sigma(x)
	=
	\text{tanh}(x)
	=
	\frac{\exp(x)-\exp(-x)}{\exp(x)+\exp(-x)},
\end{align}	
however other options are possible. 

When this particular function is used, we refer to 
\cref{eq:ann_def} as a tanh NN.
In the following, $\mathcal{N\!N}_{D,W,\ell_0,\ell_D}$
corresponds to the set of all NNs 
with input dimension $\ell_0$, output dimension $\ell_D$, a width of at most $W$, and a depth of at most $D$ layers.

\subsection{Galerkin-POD NN Architecture}
\label{sec:formulation_learning}
Let $\{\zeta^{\text{(rb)}}_1,\dots,\zeta^{\text{(rb)}}_J\}$
correspond to the basis for the finite dimensional space $X_{h,s,N,J}^{(\text{rb})}$ constructed in \cref{eq:reduced_space}
and $u_h(\y)$ the solution to \cref{eq:discrete_var_prb_a}.
Then the map
\begin{align}\label{eq:param_to_rb_coeff_map}
	\boldsymbol{\pi}^{\text{(rb)}}_{h,J}
	:
	\mathbb{U}  \to \mathbb{C}^J:
	\y \mapsto 
	\begin{pmatrix}
		\dotp{u_h(\y)}{\zeta^{\text{(rb)}}_1}_X \\
		\vdots \\
		\dotp{u_h(\y)}{\zeta^{\text{(rb)}}_J}_X
	\end{pmatrix}
\end{align}
gathers the coefficients of the projection of $u_h(\y)$ onto 
the subspace $X_{h,s,N,J}^{(\text{rb})}$,~i.e. of 
$\mathsf{P}_{X_{h,s,N,J}^{(\text{rb})}} u_h(\y)$, for each
$\y \in \mathbb{U}$. Unfortunately, as the setting under
consideration makes use of complex-valued Hilbert spaces,
the map introduced in \cref{eq:param_to_rb_coeff_map} is complex-valued and
we cannot readily use NNs as defined in \cref{sec:deep_neural_networks}.
To alleviate this, and as described in \cite[Section 4.2]{weder2024galerkin}, we consider
instead the mapping
\begin{equation}\label{eq:pi_real_output}
	\boldsymbol{\pi}^{\text{(rb)}}_{h,J,\mathbb{R}}
	:
	\mathbb{U}  \to \mathbb{R}^{2J}:
	\y
	\mapsto
	\begin{pmatrix}
		{\boldsymbol \alpha^\Re ( \boldsymbol{y})} \\
		{\boldsymbol \alpha^\Im( \boldsymbol{y})}
	\end{pmatrix}
	\coloneqq
    \begin{pmatrix}
        \Re \left\{
        		\boldsymbol{\pi}^{\text{(rb)}}_{h,J}(\y)
        \right\}  \\
        \Im \left\{ 
        		\boldsymbol{\pi}^{\text{(rb)}}_{h,J}(\y)
        \right\}
    \end{pmatrix}
    \in \IR^{2J},
    \quad
    \y \in \mathbb{U},
\end{equation}
which approximates the real and imaginary parts of the output in \cref{eq:param_to_rb_coeff_map} separately.
We observe that the maps 
\begin{equation}\label{eq:real_imag_NN}
	\mathcal{A}^\Re:
	\mathbb{U}  \to \mathbb{R}^J:
	\y \mapsto {\boldsymbol \alpha^\Re ( \boldsymbol{y})}
	\quad
	\text{and}
	\quad
	\mathcal{A}^\Im:
	\mathbb{U}  \to \mathbb{R}^J:
	\y \mapsto {\boldsymbol \alpha^\Im ( \boldsymbol{y})}
\end{equation}
are $(\boldsymbol{b},p,\varepsilon)$-holomorphic as consequence of \cite[Lemma A.1]{dolz2024parametric},
thus rendering \cref{eq:pi_real_output} so as well.

For the approximation of $\boldsymbol{\pi}^{\text{(rb)}}_{h,J,\mathbb{R}}$ we seek a tanh NN $\boldsymbol{\pi}^{\text{(rb)}}_{\boldsymbol{\theta}} \in \mathcal{N\!N}_{D,W,s,2J}$ with $\boldsymbol \theta \in \Theta$,
i.e.~with $s \in \mathbb{N}$ inputs (one for each component of the parametric input $\y \in \mathbb{U}^{(s)}$), $2J$ outputs
accounting for the $J$ complex reduced coefficients, and depth and width $D$ and $W$, respectively.
The first $J$ outputs of this NN are denoted as
${\boldsymbol \alpha^\Re_{\boldsymbol \theta}( \boldsymbol{y})}$, whereas the last $J$ by ${\boldsymbol \alpha^\Im_{\boldsymbol \theta}( \boldsymbol{y})}$.
These are intended to approximate the maps defined in
\cref{eq:real_imag_NN}. We refer to \cref{fig:NN_complex_decomp} for an illustration of this architecture.
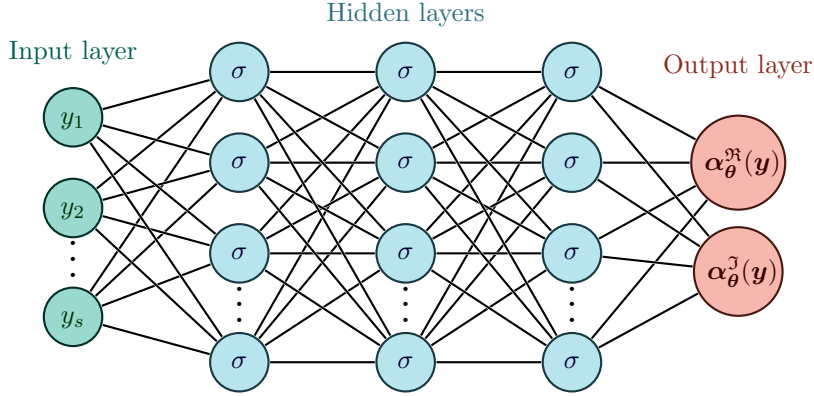
\begin{figure}[htp]
\begin{center}
\begin{tikzpicture}[x=2.2cm,y=1.2cm]
  \message{^^JNeural network, shifted}
  \readlist\Nnod{3,4,4,4,2} %
  \readlist\Nstr{s,m,m,m,L} %
  \readlist\Cstr{\strut  y,\sigma,\sigma,\sigma,{\boldsymbol \alpha (y)}} %
  \def\yshift{0.2} %
  
  \message{^^J  Layer}
  \foreachitem \N \in \Nnod{ %
    \def\lay{\Ncnt} %
    \pgfmathsetmacro\prev{int(\Ncnt-1)} %
    \message{\lay,}
    \foreach \i [evaluate={\c=int(\i==\N); \y=\N/2-\i-\c*\yshift;
    		\cu=int(\i==2); \yu=\N/2-\i-\cu*\yshift;
                 \index     = (\i<\N?int(\i):"\Nstr[\lay]");
                 \indexex = (\i<\N?int(\i):"\Nstr[\lay]");
                 \x=\lay; \n=\nstyle;}] in {1,...,\N}{ %
      \ifnum\lay=1 
      	\node[node \n] (N\lay-\i) at (\x,\y) {$\Cstr[\lay]_{\index}$};
      \fi
      \ifnum\lay>1 
      \ifnum\lay<5 
      	\node[node \n] (N\lay-\i) at (\x,\y) {$\Cstr[\lay]$};
      \fi
      \ifnum\lay=5
      	\ifnum\i=1
      	\node[node \n] (N\lay-\i) at (\x,\y) {{
	${\boldsymbol \alpha^\mathfrak{R}_{\boldsymbol \theta} ( \boldsymbol{y})}$}};
	\fi
      	\ifnum\i=2
      	\node[node \n] (N\lay-\i) at (\x,\yu) {{
	${\boldsymbol \alpha^\mathfrak{I}_{\boldsymbol \theta}(\boldsymbol{y})}$}};
	\fi
      	\ifnum\i=3
      	\node[node \n] (N\lay-\i) at (\x,\y) {{
	${\boldsymbol \alpha_{1,\mathfrak{I},\boldsymbol\theta} (\boldsymbol{y})}$}};
	\fi
      	\ifnum\i=4
      	\node[node \n] (N\lay-\i) at (\x,\y) {{
	${\boldsymbol \alpha_{L,\mathfrak{I},\boldsymbol\theta} (\boldsymbol{y})}$}};
	\fi
      \fi
      \fi
      \ifnum\lay>1 %
        \foreach \j in {1,...,\Nnod[\prev]}{ %
          \draw[connect,white,line width=1.2] (N\prev-\j) -- (N\lay-\i);
          \draw[connect] (N\prev-\j) -- (N\lay-\i);
        }
      \fi %
      
    }
    \ifnum\lay<5
    \path (N\lay-\N) --++ (0,1+\yshift) node[midway,scale=1.5] {$\vdots$};
    \fi
  }
  
  \node[above=5,align=center,mygreen!60!black] at (N1-1.90) {Input layer};
  \node[above=2,align=center,myblue!60!black] at (N3-1.90) {Hidden layers};
  \node[above=10,align=center,myred!60!black] at (N\Nnodlen-1.90) {Output layer};
  
\end{tikzpicture}
\end{center}
\caption{NN architecture for the approximation of the map
$\boldsymbol{\pi}^{\text{(rb)}}_{h,J}:\mathbb{U}  \to \mathbb{C}^J$ with the NN
$\boldsymbol{\pi}^{\text{(rb)}}_{\boldsymbol\theta}: \mathbb{U}^{(s)} \to \mathbb{R}^{2J}$
The NN accepts $s\in \mathbb{N}$ inputs in the input layer
corresponding to the components of the parametric input $\y = (y_1,\dots,y_s) \in \mathbb{U}^{(s)}$.
In addition, there are $2J$ outputs for the approximation of both the real and imaginary parts, i.e.
${\boldsymbol \alpha^\Re_{\boldsymbol \theta}( \boldsymbol{y})}$ and ${\boldsymbol \alpha^\Im_{\boldsymbol \theta}}( \boldsymbol{y})$, respectively, of the reduced coefficients.}\label{fig:NN_complex_decomp}
\end{figure}

Given an orthonormal basis $\zeta^{\text{(rb)}}_1,\dots,\zeta^{\text{(rb)}}_J$ of $X_{h,s,N,J}^{(\text{rb})}$, we define the reconstruction operator $\mathcal{R}$ for any function $\Psi\colon\mathbb{U}^{(s)}\to\mathbb{R}^{2J}$ as $\mathcal{R}(\Psi)\colon\mathbb{U}^{(s)}\to X_{h,s,N,J}^{(\text{rb})}$, with
\begin{equation}
	\mathcal{R}
	\left(
	\Psi
	\right)
	(\y)
	=
	\sum_{i=1}^{J}
	\left(
	\left(
	\Psi(\y)
	\right)_{i}
	+
	\imath
	\left(
	\Psi(\y)
	\right)_{i+J}
	\right)
	{\zeta}^{\normalfont\text{(rb)}}_i,
	\quad
	\y \in \mathbb{U}^{(s)},
\end{equation}
with $\imath$ denoting the imaginary unit. Then,
for a given $\boldsymbol\theta \in \Theta$,
the reconstruction of $\boldsymbol{\pi}^{\text{(rb)}}_{\boldsymbol\theta}$
is given by
\begin{equation}
\begin{aligned}
	u^{\text{(rb,$\mathcal{N\!N}$)}}_{J,\boldsymbol\theta}(\y)
    &
	=
	\mathcal{R}
	\left(
	\boldsymbol{\pi}^{\normalfont\text{(rb)}}_{\boldsymbol{\theta}}
	\right)
	(\y)
    \\&
    =
	\sum_{j=1}^J
	\left(
		{\boldsymbol \alpha^\mathfrak{R}_{j,\boldsymbol\theta} ( \boldsymbol{y})} 
		+
		\imath 
		{\boldsymbol \alpha^\mathfrak{I}_{j,\boldsymbol\theta} ( \boldsymbol{y})}
	\right)(\y)
	\zeta^{\text{(rb)}}_j,
	\quad
	\y \in \mathbb{U}^{(s)},
\end{aligned}
\end{equation}
where for $\y \in \mathbb{U}^{(s)}$
\begin{equation}
	{\boldsymbol \alpha^\mathfrak{R}_{\boldsymbol\theta} ( \boldsymbol{y})} 
	\coloneqq
	\begin{pmatrix}
		\boldsymbol{\alpha}^{\mathfrak{R}}_{1,\boldsymbol{\theta}}(\y) \\
		\vdots \\
		\boldsymbol{\alpha}^{\mathfrak{R}}_{J,\boldsymbol{\theta}}(\y) 
	\end{pmatrix}
	\in \mathbb{R}^J
	\quad
	\text{and}
	\quad
	{\boldsymbol \alpha^\mathfrak{I}_{\boldsymbol\theta} ( \boldsymbol{y})} 
	\coloneqq
	\begin{pmatrix}
		\boldsymbol{\alpha}^{\mathfrak{I}}_{1,\boldsymbol{\theta}}(\y) \\
		\vdots \\
		\boldsymbol{\alpha}^{\mathfrak{I}}_{J,\boldsymbol{\theta}}(\y)
	\end{pmatrix}
	\in \mathbb{R}^J.
\end{equation}
	
\subsection{Fully Discrete Error Analysis}
We present a fully discrete analysis of the Galerkin-POD NN algorithm based on the results introduced in \Cref{sec:analysis_galerkin_POD_RBM}.

For a given $s \in \mathbb{N}$, we set
\begin{equation}
	\mathcal{T}_s: \mathbb{U} \rightarrow \mathbb{U}^{(s)}:  (y_1,\ldots,y_s,y_{s+1},\ldots) \mapsto (y_1,\dots,y_s).
\end{equation}

\begin{lemma}\label{lmm:NN_bound}
Assume that $\boldsymbol{b} \in \ell^{p}(\mathbb{N})$
is strictly decreasing.
For each $n \in \mathbb{N}$, $n\geq s$, there exists a tanh NN
$\boldsymbol{\pi}^{\normalfont\text{(rb)}}_{n} \in \mathcal{N\!N}_{D,W,s,2J}$
such that
\begin{equation}
	\norm{
		\boldsymbol{\pi}^{\normalfont\text{(rb)}}_{h,J,\mathbb{R}}
		-
		\boldsymbol{\pi}^{\normalfont\text{(rb)}}
        _{n}
		\circ
		\mathcal{T}_s
	}_{L^2(\mathbb{U};\mathbb{R}^{2J})}
	\lesssim
    n^{-\left(1/p - 1/2\right)}
    +
    s^{-\left(1/p-1\right)}
\end{equation}
with 
$
	D
	=
	\mathcal{O}
	\left(
		\log_2(n)
	\right)
$
and
$
	W
	=
	\mathcal{O}(n^2)
$.
\end{lemma}

\begin{proof}
This result follows in analogy to \cite[Lemma 5.7]{dolz2024parametric}, which 
in turn uses tools from \cite{ABDM2022}, and \cref{eq:truncationerror}.
\end{proof}

Equipped with this result, together with 
the results presented in \Cref{sec:analysis_galerkin_POD_RBM},
we may state the following error bound.

\begin{theorem}\label{thm:full_error_NN}
Assume that $\boldsymbol{b} \in \ell^{p}(\mathbb{N})$
is strictly decreasing.
Then there exists $\boldsymbol{\pi}^{\normalfont\text{(rb)}}_{n}
\in \mathcal{N\!N}_{D,W,s,2J}$ such that
\begin{equation}
\begin{aligned}
	\norm{
		u
		- 
		\mathcal{R}
        \left(
            \boldsymbol{\pi}^{\normalfont\text{(rb)}}_{n}
        \right)
        \circ\mathcal{T}_s
	}_{L^2(\mathbb{U};X)}
	&
	\lesssim
    \sup_{\y \in \mathbb{U}}
    \inf_{v_h \in X_h}
    \norm{u(\y)-v_h}_X   
    +
    n^{-\left(1/p - 1/2\right)}
    \\
    &\qquad
    +
    s^{-\left(1/p-1\right)}
    +
    N^{-\frac{\alpha}{2}}
    +
    J^{-\left(1/p-1\right)}.
\end{aligned}
\end{equation}
with 
$W = \mathcal{O}(n^2)$ and $D = \mathcal{O} \left(\log_2(n)\right)$,
where for a tanh NN $\Psi$ with $s$ inputs and $2J$
outputs.
\end{theorem}

\begin{proof}
Let $\boldsymbol{\pi}^{\normalfont\text{(rb)}}_{n}$ be as in 
\Cref{lmm:NN_bound}. It follows from the application of the triangle inequality
that
\begin{equation}\label{eq:bound_RB}
\begin{aligned}
	\norm{
		u
		- 
		\mathcal{R}
        \left(
            \boldsymbol{\pi}^{\normalfont\text{(rb)}}_{n}
        \right)
        \circ\mathcal{T}_s
	}_{L^2(\mathbb{U};X)}
	&\leq
	\underbrace{
	\norm{
		u
		- 
		\mathsf{P}_{X_{h,s,N,J}^{(\text{rb})}}
		u_h
	}_{L^2(\mathbb{U};X)}
	}_{(\spadesuit)}
	\\
	&\qquad
	+
	\underbrace{
	\norm{
		\mathsf{P}_{X_{h,s,N,J}^{(\text{rb})}}
		u_h
		- 
		\mathcal{R}
        \left(
            \boldsymbol{\pi}^{\normalfont\text{(rb)}}_{n}
        \right)
        \circ\mathcal{T}_s
	}_{L^2(\mathbb{U};X)}
	}_{(\clubsuit)}.
\end{aligned}
\end{equation}
The term $(\spadesuit)$ is bounded according to \cref{cor:full_error_projection}.
We proceed to bound $(\clubsuit)$. Recalling that $\zeta^{\text{(rb)}}_1,\dots,\zeta^{\text{(rb)}}_J$
is an orthonormal basis of $X_{h,J}^{(\text{rb})}$ with respect to the inner product of $X$
one may readily observe that
\begin{equation}\label{eq:bound_NN}
	\norm{
		\mathsf{P}_{X_{h,s,N,J}^{(\text{rb})}}
		u_h
		- 
		\mathcal{R}
        \left(
            \boldsymbol{\pi}^{\normalfont\text{(rb)}}_n
        \right)
        \circ\mathcal{T}_s
	}_{L^2(\mathbb{U};X)}
	=
	\norm{
		\boldsymbol{\pi}^{\text{(rb)}}_{h,J,\mathbb{R}}
		-
		\boldsymbol{\pi}^{\text{(rb)}}_n
        \circ\mathcal{T}_s
	}_{L^2(\mathbb{U};\mathbb{R}^{2J})}.
\end{equation}
The application of \cref{lmm:NN_bound} to bound \cref{eq:bound_NN} yields the assertion.
\end{proof}

Equilibrating approximation errors yields a-priori requirements for the neural network parameters to maintain the approximation rates of the quasi-Monte Carlo sampling.
\begin{corollary}\label{cor:NNparamconv}
Assume that $\boldsymbol{b} \in \ell^{p}(\mathbb{N})$
is strictly decreasing. Select $J$ as in \cref{cor:PODrankbound} and
$n\sim N^\frac{\alpha p}{2-p}$.
Then there exists a tanh NN $\boldsymbol{\pi}^{\normalfont\text{(rb)}}_{n}
\in \mathcal{N\!N}_{D,W,s,2J}$ of depth $D=\mathcal{O}\left(\log_2(n)\right)$ and width $W=\mathcal{O}\left(n^2\right)$ such that
\begin{equation}
	\norm{
		u
		- 
		\mathcal{R}
        \left(
            \boldsymbol{\pi}^{\normalfont\text{(rb)}}_{n}
        \right)
        \circ\mathcal{T}_s
	}_{L^2(\mathbb{U};X)}
	\lesssim
    \sup_{\y \in \mathbb{U}}
    \inf_{v_h \in X_h}
    \norm{u(\y)-v_h}_X
    +
    s^{-\left(1/p-1\right)}
    +
    N^{-\frac{\alpha}{2}}.
\end{equation}
\end{corollary}

\subsection{Neural Network Training}
\label{sec:NN_training}
Having proven the existence of NN with good approximation properties, it remains to comment on how to construct a realization of such NN, a procedure commonly referred to as \emph{training}. To this end, we observe that any NN $\boldsymbol{\pi}^{\text{(rb)}}_{\boldsymbol{\theta}_N}\in \mathcal{N\!N}_{D,W,s,2J}$, with $\boldsymbol{\theta}_N \in \Theta$, satisfies
\begin{equation}\label{eq:MSEmotivation}
\begin{aligned}
	&\norm{
	u
	- 
	\mathcal{R}
	\left(
	\boldsymbol{\pi}^{\normalfont\text{(rb)}}_{\boldsymbol{\theta}_N}
	\right)
	\circ\mathcal{T}_s
	}_{L^2(\mathbb{U};X)}\\
	&\quad\qquad\leq
	\norm{
	u
	- 
	\mathcal{R}
	\left(
	\boldsymbol{\pi}^{\text{(rb)}}_{h,J,\mathbb{R}}
	\right)
	\circ\mathcal{T}_s
	}_{L^2(\mathbb{U};X)}
        +
	\norm{
	\mathcal{R}
	\left(
	\boldsymbol{\pi}^{\text{(rb)}}_{h,J,\mathbb{R}}
	-
	\boldsymbol{\pi}^{\normalfont\text{(rb)}}_{\boldsymbol{\theta}_N}
	\circ\mathcal{T}_s
	\right)
	}_{L^2(\mathbb{U};X)}\\
	&\quad\qquad\leq
	\underbrace{\norm{
	u
	- 
	\mathcal{R}
	\left(
	\boldsymbol{\pi}^{\text{(rb)}}_{h,J,\mathbb{R}}
	\right)
	\circ\mathcal{T}_s
	}_{L^2(\mathbb{U};X)}}_{=(\spadesuit)}+
	\underbrace{\norm{
		\boldsymbol{\pi}^{\text{(rb)}}_{h,J,\mathbb{R}}
		-
		\boldsymbol{\pi}^{\normalfont\text{(rb)}}_{\boldsymbol{\theta}_N}
	}_{L^2(\mathbb{U}^{(s)};\mathbb{R}^{2J})}}_{=(\clubsuit)},
\end{aligned}
\end{equation}
where $\boldsymbol{\pi}^{\text{(rb)}}_{h,J,\mathbb{R}}$ is as in \cref{eq:pi_real_output}. While $(\spadesuit)$ is estimated using \cref{cor:PODrankbound}, in analogy to
\cref{lem:QMCerror}, we consider
an approximation for $(\clubsuit)$ of the form
\begin{equation}\label{eq: MSE_loss_nn}
	L_{\text{MSE}}(\boldsymbol\theta_N)
	\coloneqq
	\frac{1}{N}\sum_{i = 1}^{N}
	\norm{
		\boldsymbol{\pi}^{\text{(rb)}}_{h,J,\mathbb{R}}
		\left(\y^{(i)}\right) 
		- 
		\boldsymbol{\pi}^{\text{(rb)}}_{\boldsymbol{\theta}_N}\left(\y^{(i)}\right)
	}_{\mathbb{R}^{2J}}^2,
\end{equation}
where $\y^{(i)} \in \mathbb{U}^{(s)}, i = 1, \ldots, N$, are training inputs with $\boldsymbol{\pi}^{\text{(rb)}}_{h,J,\mathbb{R}}
\left(\y^{(i)}\right)$ being the corresponding high-fidelity snapshots $u_h\left(\y^{(i)}\right), i = 1, \ldots, N$, projected on the POD basis.
$L_{\text{MSE}}$ is known as the \emph{mean squared error (MSE)}.

As per customary, in the following 
we work under the assumption that
\begin{equation}\label{eq:MSEassumption}
	\norm{
		\boldsymbol{\pi}^{\text{(rb)}}_{h,J,\mathbb{R}}
		-
		\boldsymbol{\pi}^{\normalfont\text{(rb)}}_{\boldsymbol{\theta}_N}
	}_{L^2(\mathbb{U}^{(s)};\mathbb{R}^{2J})}
	\approx
	\sqrt{L_{\text{MSE}}(\boldsymbol\theta_N)},
\end{equation}
and we focus on adjusting the parameters $\boldsymbol{\theta}_N$ such that $L_{\text{MSE}}(\boldsymbol \theta_N)$ is minimized. In fact, the mean squared error is one of the most common choices for NN training and readily implemented in many software packages. While solving the corresponding optimization problem is known to be difficult, our analysis provides us at least with a sufficient stopping criterion for optimization. More precisely, stopping the optimization procedure when
\begin{equation}\label{eq:stopping_criterion}
	\sqrt{L_{\text{MSE}}(\boldsymbol \theta)}
	\lesssim
	\sup_{\y \in \mathbb{U}}
	\inf_{v_h \in X_h}
	\norm{u(\y)-v_h}_X
	+
	s^{-\left(\frac{1}{p}-1\right)}
	+
	N^{-\frac{\alpha}{2}}
\end{equation}
and assuming \cref{eq:MSEassumption} implies with \cref{eq:MSEmotivation} that
\begin{equation}\label{eq:desired_nn_estimate}
	\norm{
		u
		- 
		\mathcal{R}
		\left(
		\boldsymbol{\pi}^{\normalfont\text{(rb)}}_{\boldsymbol{\theta}_N}
		\right)
		\circ\mathcal{T}_s
	}_{L^2(\mathbb{U};X)}
	\lesssim
	\sup_{\y \in \mathbb{U}}
	\inf_{v_h \in X_h}
	\norm{u(\y)-v_h}_X
	+
	s^{-\left(\frac{1}{p}-1\right)}
	+
	N^{-\frac{\alpha}{2}}.
\end{equation}

\begin{remark}
In our numerical experiments below we observe that \cref{eq:stopping_criterion} does not always imply \cref{eq:desired_nn_estimate}, especially when $N$ is large. This indicates that the common assumption \cref{eq:MSEassumption} needs further consideration to close the gap between theory and practice.
Indeed, for the approximation stated
in \cref{eq:MSEassumption} to be valid up to a
prescribed accuracy depending upon the total number of samples $N$ when using QMC points, 
one needs to study the $(\boldsymbol{b},p,\epsilon)$-holomorphy property of the map 
$\y \mapsto \boldsymbol{\pi}^{\normalfont\text{(rb)}}_{\boldsymbol{\theta}_N}$ for a given configuration of weights 
$\boldsymbol{\theta}_N \in \Theta$. This has been thoroughly addressed in \cite{LMR2020}
and more recently in \cite{keller2025regularity}, where conditions on the NN weights are established for the aforementioned property to hold. Whether these conditions are compatible with existing neural network approximation results such as \cref{lmm:NN_bound} remains to be clarified.
\end{remark}

\section{Application: Sound-soft Acoustic Scattering}
\label{eq:sound_soft_scattering}
We consider a concrete application that fits the framework of \cref{sec_rom}: The scattering by a parametrically defined
sound-soft object in three spatial dimensions. The following section recapitulates the notation and main result of \cite{dolz2024parametric}.

\subsection{Parametrized Domain Deformations}
Let $\widehat{\D}\subset\IR^3$  be a bounded reference domain
with Lipschitz boundary $\widehat\Gamma=\partial\widehat{\D}$ and set $\br_\y\colon\widehat{\Gamma}\rightarrow\IR^3$ with
\begin{equation}\label{eq:affine_parametric_representation}
\br_\y(\bxref)
=
\bvarphi_0(\bxref)
+
\sum_{j\geq 1}
y_j
\bvarphi_j(\bxref),
\quad
\bxref \in \widehat\Gamma,
\quad
\y
=
(y_j)_{j\geq1}
\in 
\mathbb{U},
\end{equation}
and $\bvarphi_j\colon\widehat{\Gamma}\to\IR^3$ for $j\in \IN$.
This gives rise to a collection of parametric boundaries
$\left\{\Gamma_\y\right\}_{\y \in \mathbb{U}}$ of the form 
\begin{equation}\label{eq:parametric_boundary}
\Gamma_{\y} 
\coloneqq 
\{
\bx\in \IR^3:\; 
\bx
= 
\br_\y(\bxref),
\quad
\bxref \in \widehat\Gamma
\},
\end{equation}
In the following, we work under the assumptions stated below.

\begin{assumption}\label{assump:parametric_boundary}
	Let $\widehat\Gamma$ be the reference Lipschitz boundary.
	\begin{enumerate}
		\item\label{assump:parametric_boundary1}
		The functions $(\bvarphi_i)_{i\in \IN} \subset \mathscr{C}^{0,1}(\widehat\Gamma; \IR^3)$
		are such that for each $\y\in\mathbb{U}$ one has that
		$\br_\y\colon\widehat{\Gamma}\to\Gamma_\y$ is bijective and bi-Lipschitz, and $\Gamma_\y$ is the boundary of a Lipschitz domain.
		\item\label{assump:parametric_boundary2}
		There exists $p\in(0,1)$ such that 
		$
		\boldsymbol{b}
		\coloneqq
		\left(
		\norm{\bvarphi_j}_{\mathscr{C}^{0,1}(\widehat{\Gamma},\IR^3)} 
		\right)_{j\in \IN}
		\in 
		\ell^p(\IN).
		$
		\item\label{assump:parametric_boundary3}
		There is a decomposition $\mathcal{G}$ of $\widehat{\Gamma}$ such that
		for each $\y \in \mathbb{U}$
		and each ${\tau} \in \mathcal{G}$ one has that
		$\br_\y\circ\chi_\tau\in \mathscr{C}^{1,1}(\tilde{\tau};\IR^3)$.
	\end{enumerate}
\end{assumption}
\Cref{assump:parametric_boundary1} and \cref{assump:parametric_boundary3} of the assumption guarantee that all parametric boundaries $\Gamma_\y$ are Lipschitz and piecewise $\mathscr{C}^{1,1}$. 
The bijectivity of the boundary transformations implies that each $\Gamma_\y$ is the boundary of a parametrized domain $\D_\y$ which has the same genus as $\widehat{\D}$.
Moreover, \cref{assump:parametric_boundary2} implies absolute convergence of \cref{eq:affine_parametric_representation} as an element of $\mathscr{C}^{0,1}$. A further consequence is that the pullback operator, defined as 
$
\tau_\y \varphi
\coloneqq
\varphi\circ\boldsymbol{r}_\y
\in L^2(\widehat\Gamma),
$
for each $\y \in \mathbb{U}$ and $\varphi \in L^2(\Gamma_\y)$ is an isomorphism, i.e., 
$\tau_\y\in\mathscr{L}_{\normalfont{\text{iso}}}\left(L^2(\Gamma_\y),L^2(\widehat{\Gamma})\right)$
for each $\y \in \mathbb{U}$, see, e.g, \cite[Lemma 2.13]{dolz2024parametric}.

\subsection{Application: Sound-Soft Acoustic Scattering}
In the following, we denote by $\D_\y\subset\IR^3$ 
the domain enclosed by $\Gamma_\y$ and by $\D_\y^\cc \eqqcolon \mathbb{R}^3 \backslash\overline{\D_\y}$ 
the corresponding exterior domain. 

Provided a wavenumber $\kappa>0$ and an incident direction $\boldsymbol{\widehat{d}}_{\text{inc}} \in \mathbb{S}^2
\coloneqq \{\bx \in \IR^3: \norm{\bx}=1\}$, 
we define an incident plane wave $u^\text{inc}(\bx)\coloneqq\exp(\imath
\kappa \bx \cdot \boldsymbol{\widehat{d}}_{\text{inc}})$.
The aim is then to find the sound-soft scattered wave $u_\y^\text{scat}\in H^1_\loc(\D^\cc)$
such that the total field $u_\y\coloneqq u^\text{inc}+ u_\y^\text{scat}$ satisfies
\begin{subequations}\label{eq:scattering}
\begin{align}
	\Delta u_\y + \kappa^2 u_\y &=0, 
    \quad \text{in } \D_\y^\cc,\label{eq:sound_soft}
	\\
	u_\y &= 0, \quad \text{on } \Gamma_\y,
\end{align}
\end{subequations}
and the scattered field additionally satisfies the Sommerfeld
radiation condition
\begin{align}\label{eq:Sommerf}
    \frac{\partial u^{\text{scat}}_\y}{\partial r}(\bx) 
    - 
    \imath \kappa u^{\text{scat}}_\y(\bx) = o \left( r^{-1}\right)
\end{align}
as $r\coloneqq \norm{\bx}\rightarrow \infty$, uniformly in $\widehat{\bx}:= \bx/r$.
Thus, since $u^{\text{inc}}$ satisfies \cref{eq:sound_soft} on its own,
\cref{eq:scattering} can be cast as follows: find $u_\y^{\text{scat}}\in H_\loc^1(\D^\cc)$ such that
\begin{subequations}\label{eq:sound_soft_problem}
\begin{align}
\Delta u_\y^\text{scat} + \kappa^2 u_\y^\text{scat} &=0,\quad\qquad\text{in } \D_\y^\cc, 
\label{eq:HHeqn}
\\
u_\y^\text{scat} &=-u^\text{inc}, \quad\text{on } \Gamma_\y, 
\label{eq:SoundSftBc}
\end{align}
\end{subequations}
and \cref{eq:Sommerf} holds. \Cref{eq:sound_soft_problem} has a unique solution, which may be obtained in terms of a boundary integral formulation as outlined in the following.

\subsection{Boundary Integral Formulation}\label{sec:BIF}
Standard results yield the following representation for
the scattered field $u_\y^\text{scat}: \D_\y^\cc \rightarrow \IC$ in
terms of the unknown Neumann datum 
\begin{align}\label{eq:int_rep_form1}
u_\y^{\text{scat}}(\bx)
=
-
\mathcal{S}^{(\kappa)}_{\Gamma_\y} 
\left(
\frac{\partial u_\y}{\partial\mathbf{n}_{\Gamma_\y}}
\right)(\bx),
\quad
\bx \in \D_\y^\cc,
\end{align}
with $\mathcal{S}_{\Gamma_\y}^{(\kappa)}: H^{-\frac{1}{2}}(\Gamma_\y) \rightarrow
H^1_\loc(\Delta,\D_\y^c)$ being the acoustic single layer potential
(for further details we refer to \cite{SS10}).
Define the Dirichlet and Neumann traces onto $\Gamma_{\y}$ as
\begin{equation}
\gamma_{\Gamma_\y}^\cc\colon H^1_\loc(\D_\y^\cc)\to H^{\half}({\Gamma_\y})
\quad
\text{and}
\quad
\frac{\partial}{\partial\mathbf{n}_{\Gamma_\y}}\colon H^1(\Delta,\D_\y^\cc)\to H^{-\half}({\Gamma_\y}),
\end{equation}
which applied to \cref{eq:int_rep_form1} yield
\begin{subequations}
	\begin{align}
	\OV^{(\kappa)}_{\Gamma_\y} 
	\frac{\partial u_\y}{\partial\mathbf{n}_{\Gamma_\y}} 
	&= 
	\gamma^\cc_{\Gamma_\y} u^\text{inc}, 
	\quad \text{and} \; \Gamma_\y,\quad \text{and},  \label{eq:BIE1}\\
	\left( \half \mathsf{Id} + \OK^{(\kappa)'}_{\Gamma_\y}\right) 
	\frac{\partial u_\y}{\partial\mathbf{n}_{\Gamma_\y}} 
	&= 
	\frac{\partial u^\text{inc}}{\partial\mathbf{n}_{\Gamma_\y}},  
	\quad \text{on} \; \Gamma_\y,  \label{eq:BIE2}
	\end{align}
\end{subequations}
with the single layer operator
\begin{align}\label{eq:SLO}
\OV^{(\kappa)}_{\Gamma_\y}\isdef\gamma_{\Gamma_\y}^\cc\mathcal{S}^{(\kappa)}_{\Gamma_\y}\colon H^{-\half}(\Gamma_\y)\to H^{\half}(\Gamma_\y),
\end{align}
and the adjoint double layer operator
\begin{align}\label{eq:ALO}
\half \mathsf{Id}+\OK^{(\kappa)'}_{\Gamma_\y}\isdef\frac{\partial}{\partial\mathbf{n}_{\Gamma_\y}}\mathcal{S}^{(\kappa)}_{\Gamma_\y}\colon H^{-\half}(\Gamma_\y)\to H^{-\half}(\Gamma_\y).
\end{align}

\subsection{Combined Boundary Integral Formulation}
\label{sec:A_combined}
Given a {\it coupling parameter} $\eta \in \IR \backslash \{0\}$
we combine \cref{eq:BIE1} and \cref{eq:BIE2} to define
\begin{align}\label{eq:directBIO}
\OA^{(\kappa,\eta)'}_{\Gamma_\y}
\coloneqq
\half\mathsf{Id} 
+ 
\OK^{(\kappa)'}_{\Gamma_\y}
- 
\imath \eta \OV^{(\kappa)}_{\Gamma_\y}.
\end{align}
Exploiting that $\phi_\y\isdef\frac{\partial u_\y}{\partial\mathbf{n}_{\Gamma_\y}} \in L^2(\Gamma_\y)$, 
see, e.g., \cite{Nec1967}, a new boundary integral
approach to \cref{eq:sound_soft} is to solve for $\phi_\y \in L^2(\Gamma_\y)$ such that
\begin{align}\label{eq:A_combined}
	\OA^{(\kappa,\eta)'}_{\Gamma_\y}
	\phi_\y
	=
	f_\y
	\isdef
	\frac{\partial u^{\normalfont\text{inc}} }{\partial\mathbf{n}_{\Gamma_\y}} 
	-
	\imath \eta
	\gamma^\cc_{\Gamma_\y}
	u^{\normalfont\text{inc}}
	\in 
	L^{2}(\Gamma_\y).
\end{align}
The operator $\OA^{(\kappa,\eta)'}_{\Gamma_\y}\colon L^2(\Gamma_\y)\to L^2(\Gamma_\y)$ is a boundedly invertible and continuous linear operator for any $\kappa \in \IR_{+}$, unlike the first and second kind BIEs \cref{eq:BIE1} and \cref{eq:BIE2}, respectively, further rendering \cref{eq:A_combined} well-posed in $L^2(\Gamma_\y)$. 

We recall that the \emph{parameter-to-solution map}
\begin{align}\label{eq:parameter-to-solution}
\mathbb{U}\to L^2(\widehat\Gamma)\colon\y\mapsto\widehat{\phi}_\y \isdef\tau_\y\big(\phi_\y\big),
\end{align}
the \emph{parameter-to-operator map}
\begin{equation}\label{eq:parameter-to-operator}
\mathbb{U}
\rightarrow
\mathscr{L}
\left(
L^2(\widehat{\Gamma}),
L^2(\widehat{\Gamma})	
\right)
\colon
\y
\mapsto
\widehat{\mathsf{A}}_\y^{(\kappa,\eta)'}
\coloneqq
\tau_\y
\;
\mathsf{A}_{\Gamma_\y}^{(\kappa,\eta)'}
\tau^{-1}_\y,
\end{equation}
and the \emph{parameter-to-inverse-operator map}
\begin{equation}\label{eq:parameter-to-inverse-operator}
\mathbb{U}
\rightarrow
\mathscr{L}
\left(
L^2(\widehat{\Gamma}),
L^2(\widehat{\Gamma})	
\right)
\colon
\y
\mapsto
\Big(\widehat{\mathsf{A}}_\y^{(\kappa,\eta)'}\Big)^{-1}
=
\tau_\y
\;
\Big(\mathsf{A}_{\Gamma_\y}^{(\kappa,\eta)'}\Big)^{-1}
\tau^{-1}_\y,
\end{equation}
are $(\boldsymbol{b},p,\varepsilon)$-holomorphic and continuous if \cref{assump:parametric_boundary} is satisfied, see \cite[Lemma 2.4.2 and Corollaries 4.7 and 4.8]{dolz2024parametric}. Thus, setting $X=L^2(\widehat\Gamma)$, and
\begin{align}
\mathsf{a}(u,v;\y)
&=
\Big(\widehat{\mathsf{A}}_\y^{(\kappa,\eta)'}u,v\Big)_{L^2(\widehat{\Gamma})}\in B(L^2(\widehat\Gamma)),\label{eq:parametricblf}\\
\mathsf{f}(v;\y)
&=
(\tau_\y f_\y,v)_{L^2(\widehat{\Gamma})}\in L^2(\widehat\Gamma)\label{eq:parametriclf},
\end{align}
$\widehat{\phi_\y}$ is a solution to \cref{eq:continuous_var_prb_a} with $\mathsf{a}$ and $\mathsf{f}$ satisfying \cref{eq:lcontinuity}, \cref{eq:acontinuity}, and \cref{eq:ainfsup}.
Moreover, for a sequence of finite-dimensional, one parameter subspaces $\{\widehat{X_h}\}_{h>0}\subset L^2(\widehat\Gamma)$ that are densely embedded in $L^2(\widehat{\Gamma})$, one can show that there exists $h_0>0$ such that the discrete inf-sup condition \cref{eq:ainfsupdisc} holds for all $h\leq h_0$ see,~e.g., \cite{Kre2014}. Thus, all the derived results in \cref{sec:analysis_galerkin_POD_RBM} and \cref{sec:fully_discrete_GPOD_NN} apply to this case,
in particular \cref{cor:PODrankbound}
and \cref{cor:NNparamconv}.

\begin{remark}
Although here we rely on the results from \cite{dolz2024parametric}, one can certainly
extend these results to boundary integral formulations for open arcs following \cite{PHJ23}, two-dimensional 
boundary integral operators in fractional Sobolev spaces as in \cite{henriquez2021shape,HS21,DLM22},
and boundary integral formulations for time-dependent parabolic problems \cite{DL23}. 
\end{remark}

\begin{remark}
In the computational implementation of \cref{eq:parametricblf} and \cref{eq:parametriclf}, one may choose to define subspaces $X_{h,\y}=\tau_\y^{-1}\widehat{X_h}\subset L^2(\Gamma_\y)$ for which \cref{assump:parametric_boundary} yields a parameter-dependent sequence of one-parameter finite-dimensional subspaces which are densely embedded in $L^2(\Gamma_\y)$. The computations can then be carried out in the physical domain by using readily available software packages with a subsequent pullback of the solution to the reference domain.
\end{remark}

\section{Numerical Experiments}
\label{sec:numerical_experiment}
\subsection{Model problem}\label{sec:model}
For the numerical experiments, we consider the
sound-soft acoustic scattering problem posed on three-dimensional parametric Lipschitz boundaries, exactly as discussed in \cref{eq:sound_soft_scattering}.
Our goal is to learn the corresponding parameter-to-solution map \cref{eq:parameter-to-solution} following \cref{sec:formulation_learning}.

To this end, we choose our reference boundary $\widehat\Gamma$ to be the three-dimensional turbine geometry portrayed in \cref{fig:turbine}.
\begin{figure}
    \centering
    \includegraphics[width=0.6\linewidth,clip=true,trim=780 480 600 320]{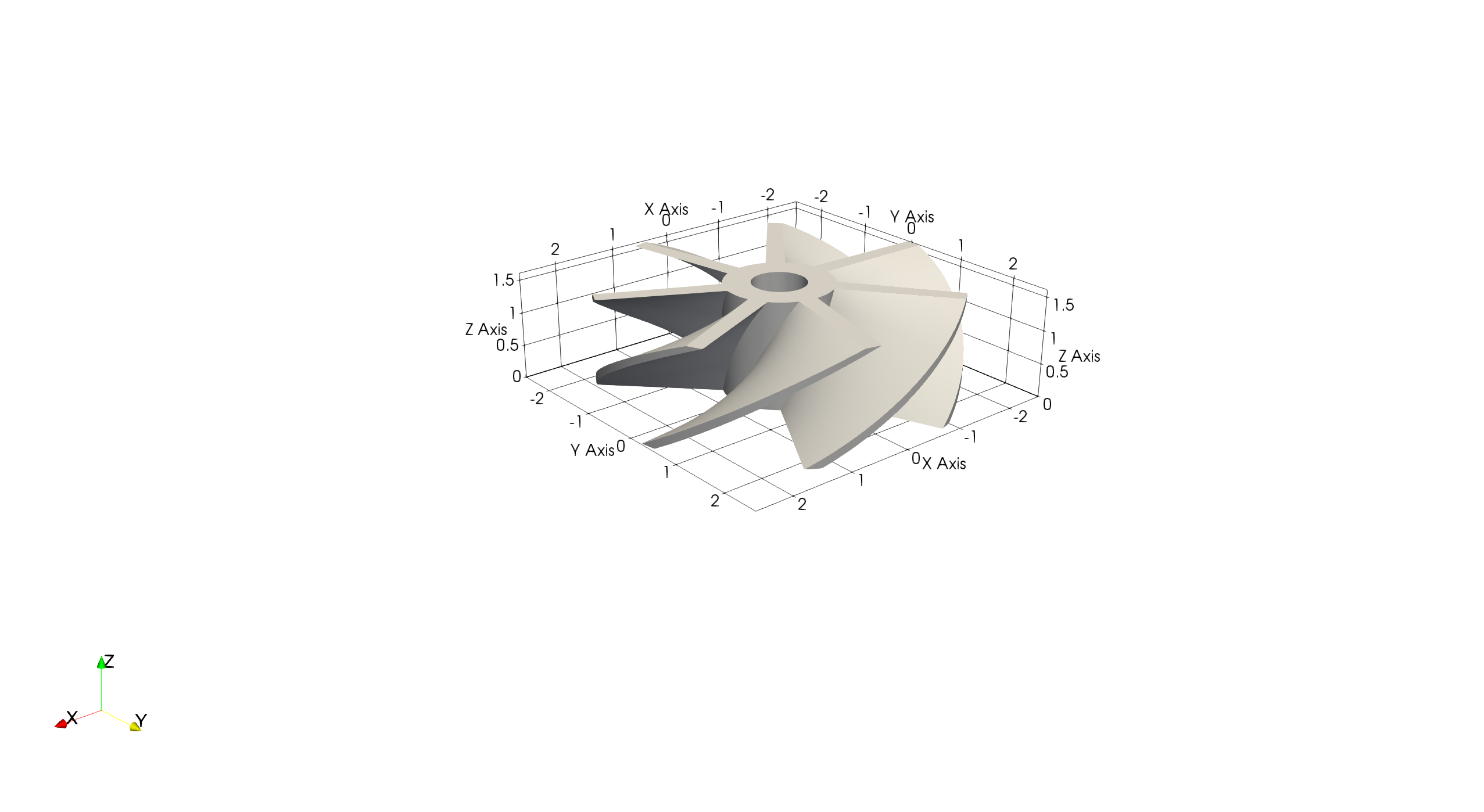}
    \caption{The turbine geometry which is randomly deformed.}
    \label{fig:turbine}
\end{figure}
For the domain deformations $\br_\y\colon\widehat{\Gamma}\rightarrow\IR^3$ we choose a scaled version of \cref{eq:affine_parametric_representation} with $\bvarphi_0(\bxref)=\bxref$ and $\bvarphi_i(\bxref)=\sqrt{\lambda_k}\bchi_k(\bxref)$, $k\in\mathbb{N}$, where $(\lambda_k,\bchi_k)$ are the eigenpairs of the covariance operator $\mathcal{C}\colon \big[L^2(\widehat\Gamma)\big]^3\to \big[L^2(\widehat\Gamma)\big]^3$ given by
\begin{align}\label{eq:surfacecovarianceoperator}
(\mathcal{C}\bu)(\bxref)\isdef\int_{\widehat\Gamma}\Cov[\br](\bxref,\yref)\bu(\yref)\text{d}\sigma_{\yref}
\end{align}
with
\[
\Cov[\br](\bxref,\yref)
=
\begin{pmatrix}
\frac{4}{5}K_{7/2}\Big(\frac{2\|\bxref-\yref\|_2}{3}\Big) & \frac{1}{10}K_{7/2}\Big(\frac{\|\bxref-\yref\|_2}{8}\Big) & 0 \\
\frac{1}{10}K_{7/2}\Big(\frac{\|\bxref-\yref\|_2}{8}\Big) & \frac{2}{5}K_{7/2}\Big(\frac{\|\bxref-\yref\|_2}{6}\Big) & 0 \\
0 & 0 & \frac{2}{10}K_{7/2}\Big(\frac{\|\bxref-\yref\|_2}{24}\Big)
\end{pmatrix}.
\]
Here, $K_{7/2}\colon\mathbb{R}_{>0}\to\mathbb{R}_{>0}$ refers to the Mat\'ern $7/2$-kernel, which is given as
\[
K_{7/2}(r)=\bigg(1 + 3r + \frac{27r^2}{7} + \frac{18r^3}{7} + \frac{27r^4}{35}\bigg) e^{-3r}
\]
and implies $p=4/9$, cf., e.g., \cite{GKN+2015}.
A few examples of parametrized domain boundaries generated by these domain deformations are shown in \cref{fig:turbine_deformed}.
\begin{figure}
    \centering
    \includegraphics[width=0.3\linewidth,clip=true,trim=1000 400 900 400]{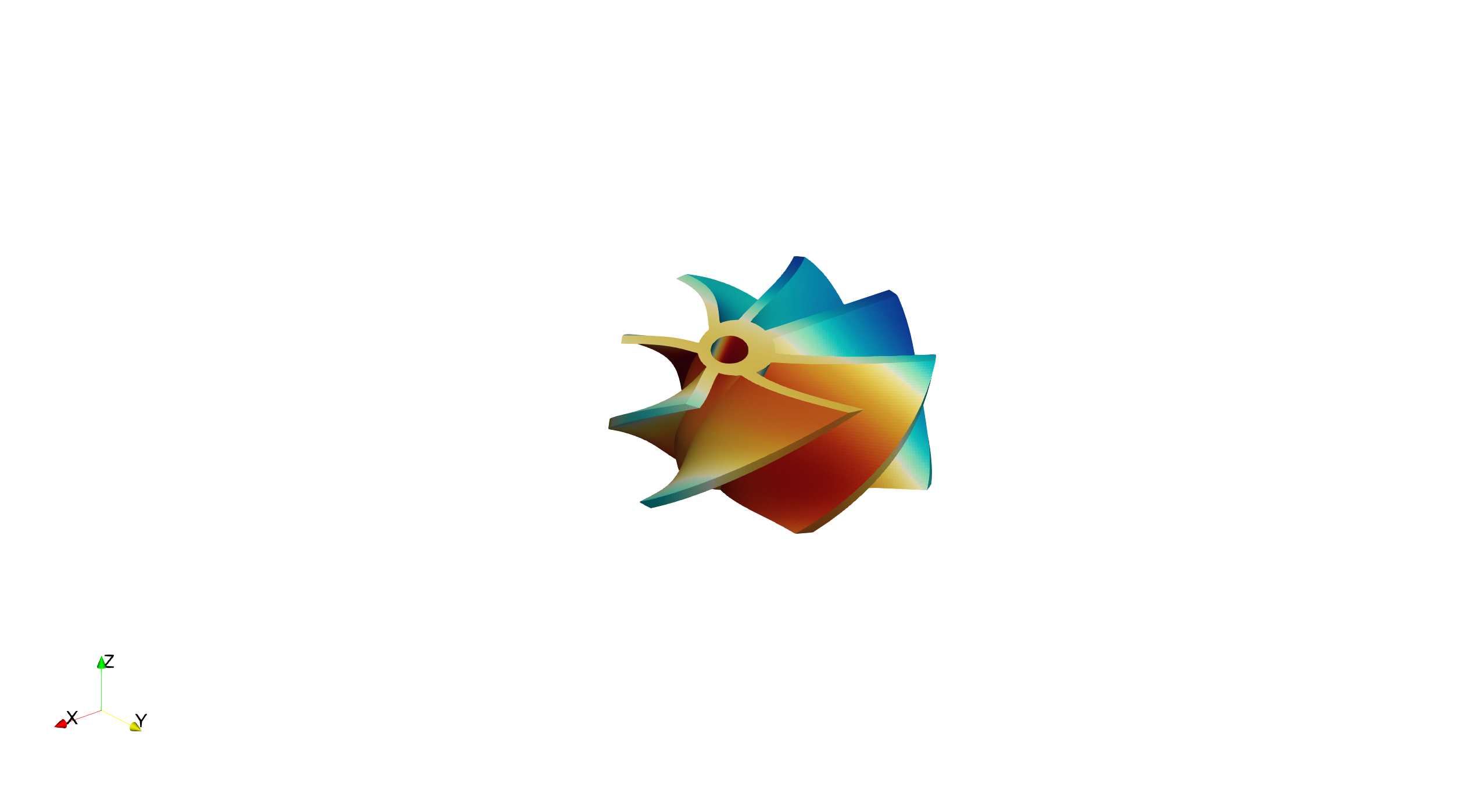}
    \includegraphics[width=0.3\linewidth,clip=true,trim=1000 450 700 250]{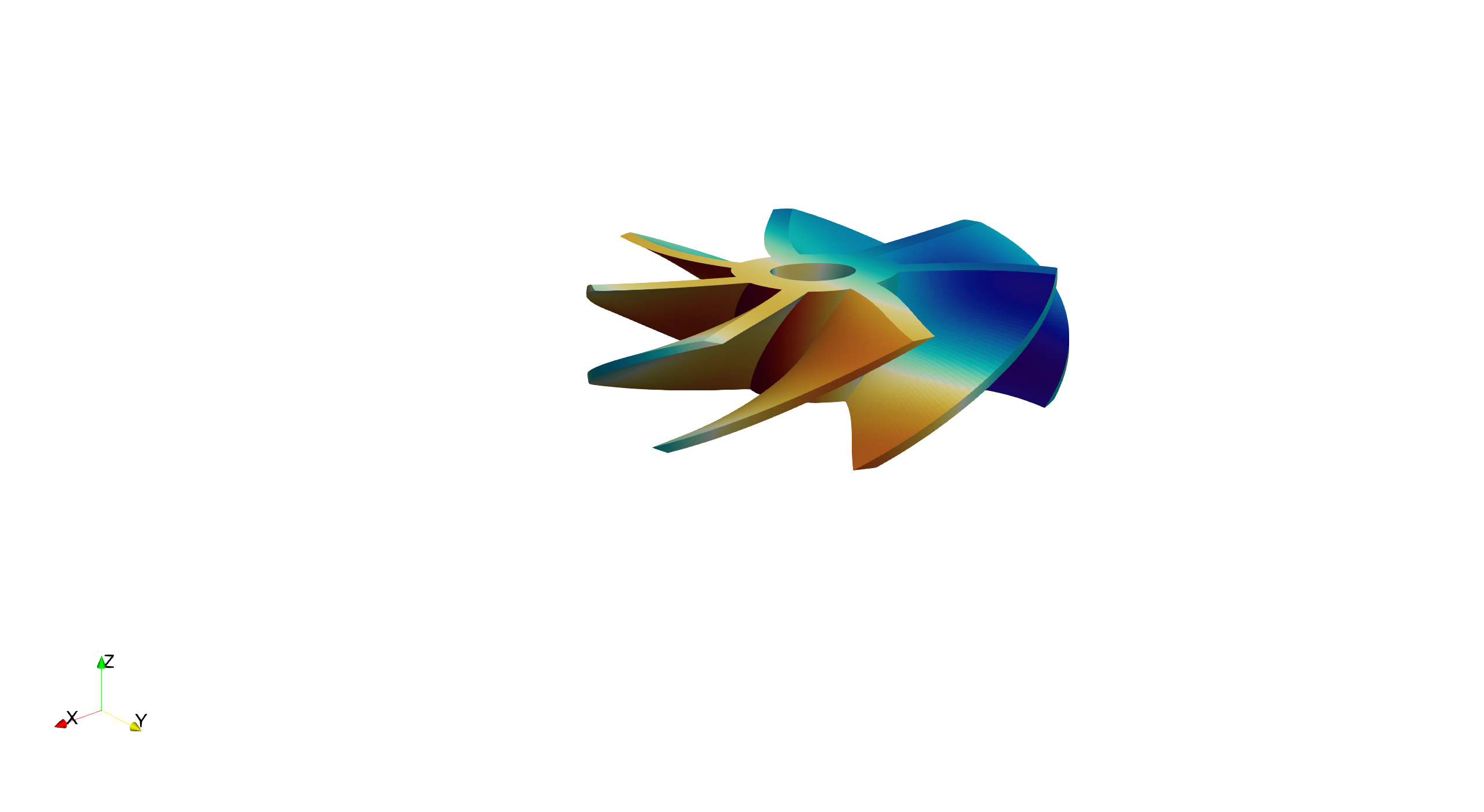}
    \includegraphics[width=0.3\linewidth,clip=true,trim=1100 550 800 300]{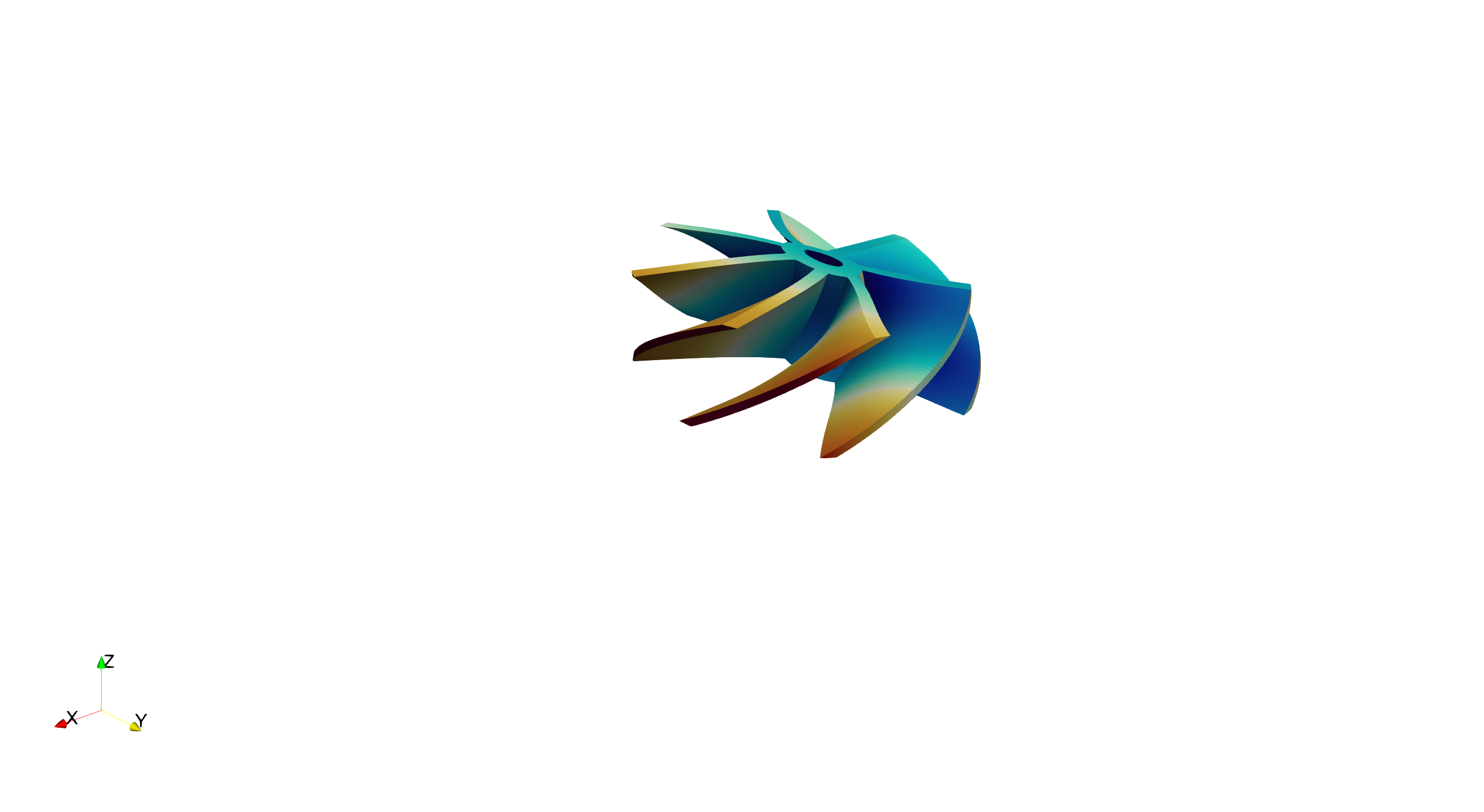}\\
    \includegraphics[width=0.3\linewidth,clip=true,trim=1000 400 900 400]{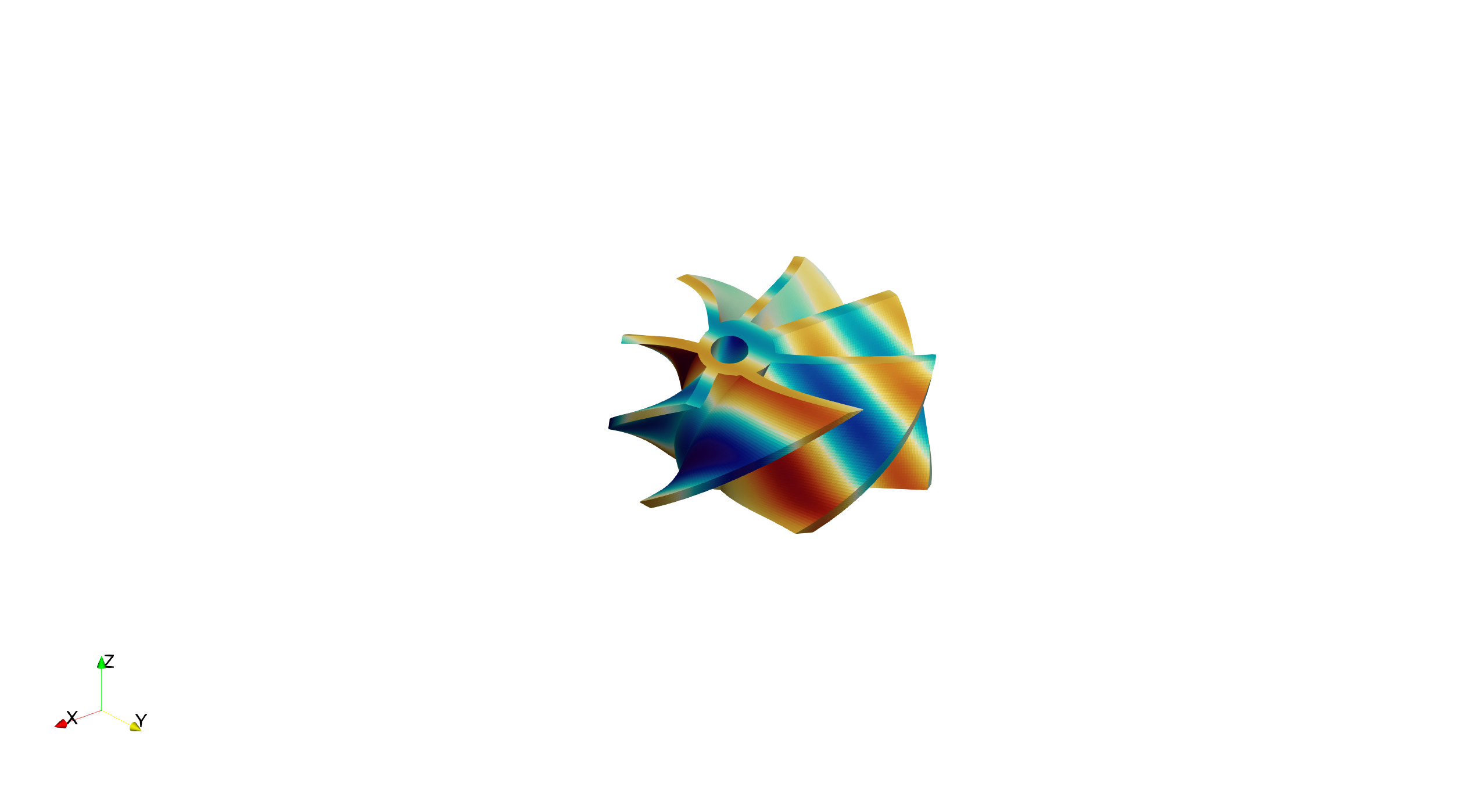}
    \includegraphics[width=0.3\linewidth,clip=true,trim=1000 450 700 250]{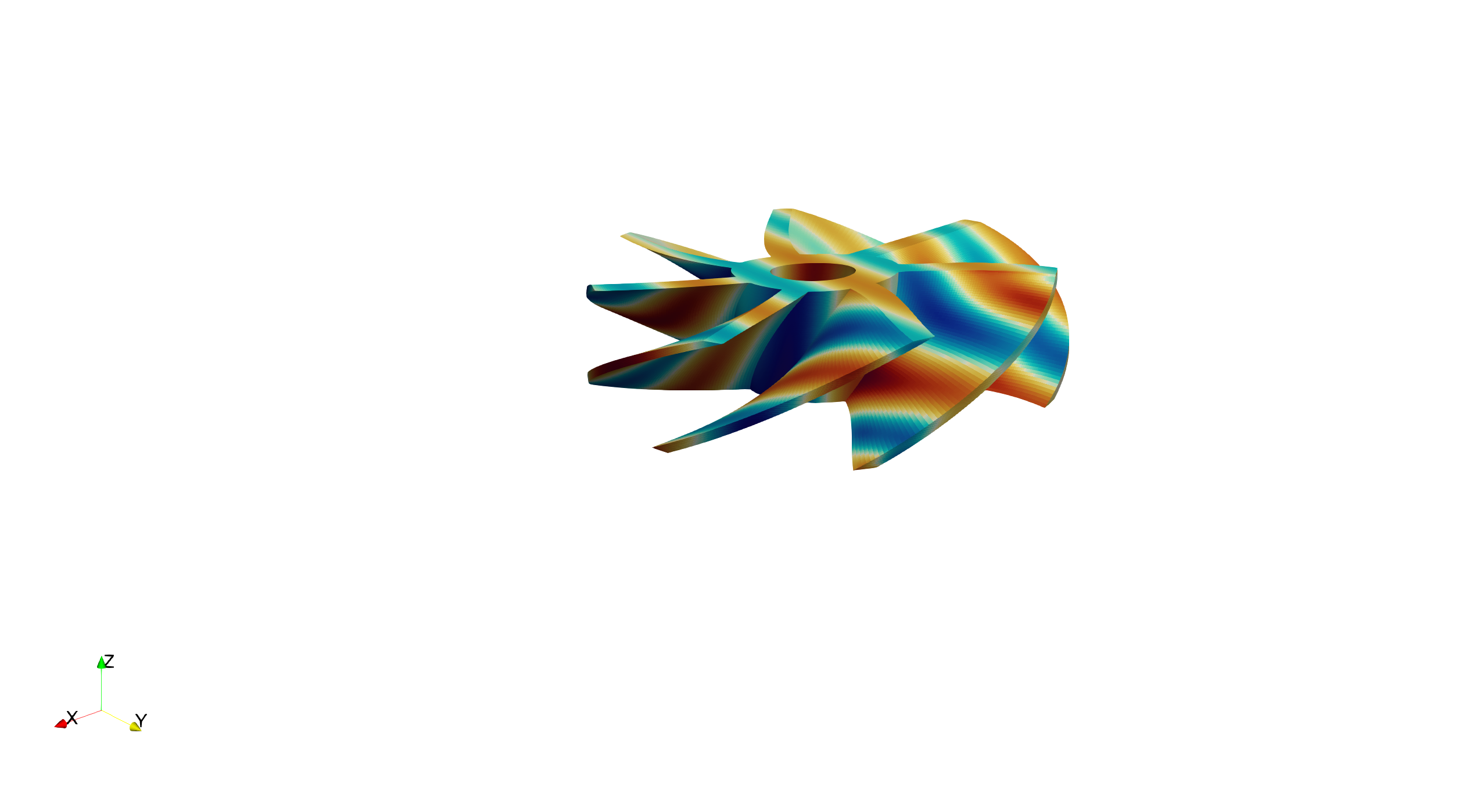}
    \includegraphics[width=0.3\linewidth,clip=true,trim=1100 500 800 300]{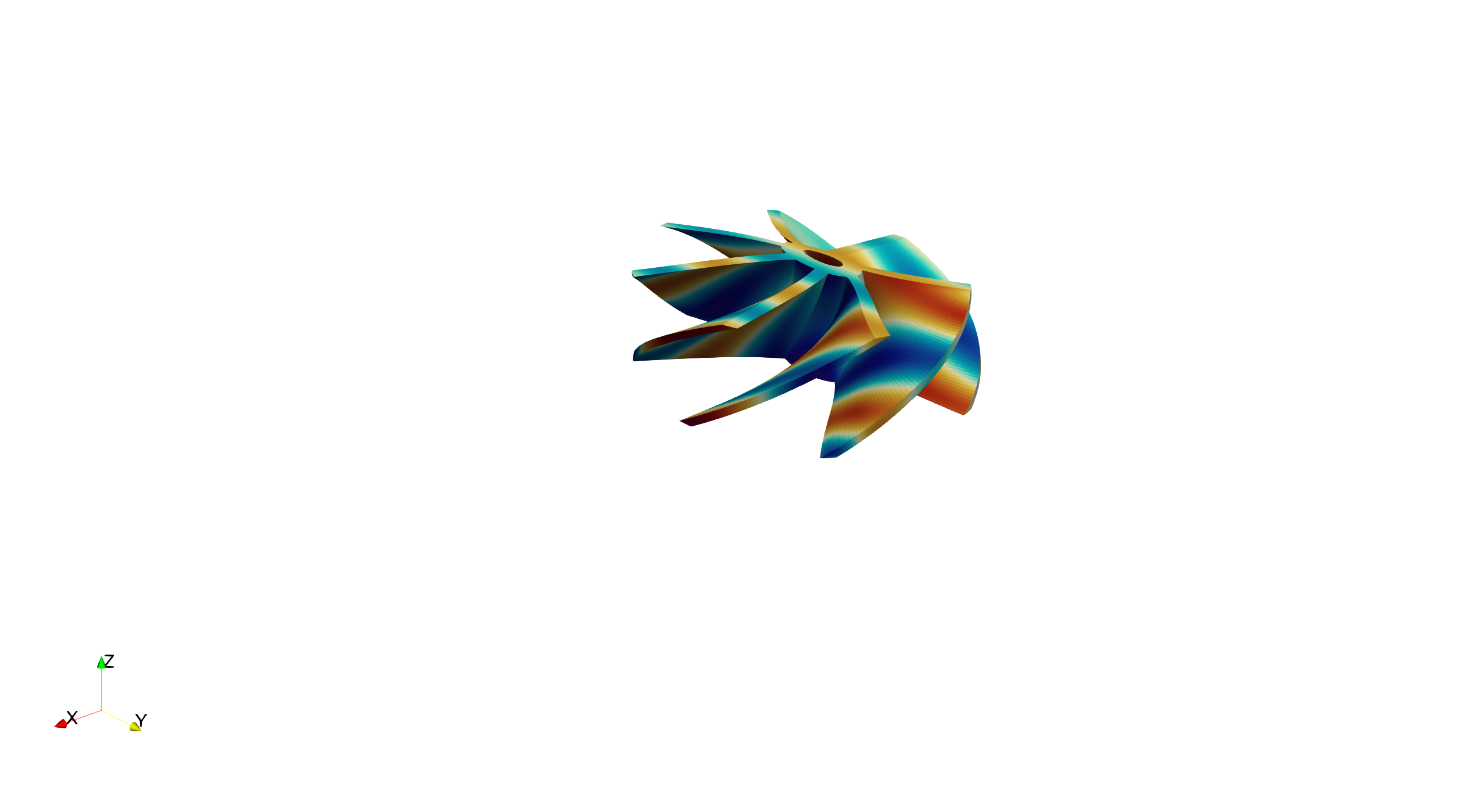}
    \caption{Realizations of the randomly deformed domain. Colors illustrate the real part of the values of the parameter-to-solution map given by \cref{eq:A_combined} for the wavenumbers $\kappa=1$ (top) and $\kappa=4$ (bottom).}
    \label{fig:turbine_deformed}
\end{figure}
The decay of the eigenvalues of the covariance operator \cref{eq:surfacecovarianceoperator} is depicted in \cref{fig:eigenvaluedecay}.
\begin{figure}
    \centering
    \includegraphics[height=0.45\textwidth]{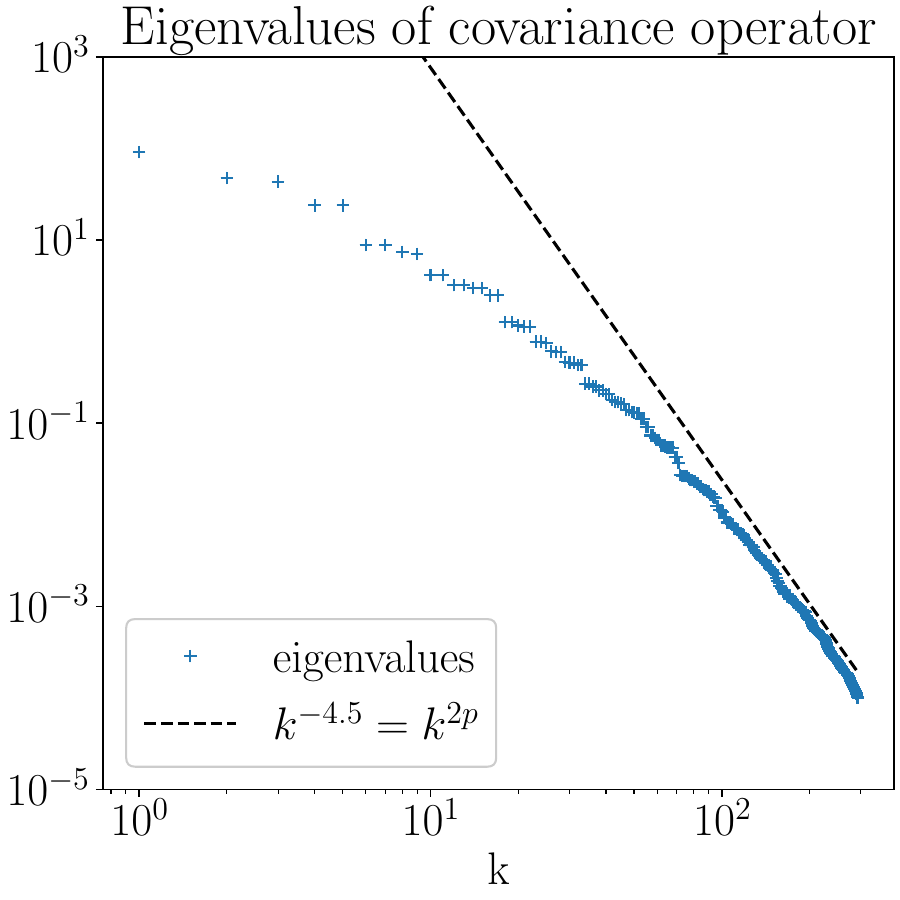}
    \includegraphics[height=0.45\textwidth]{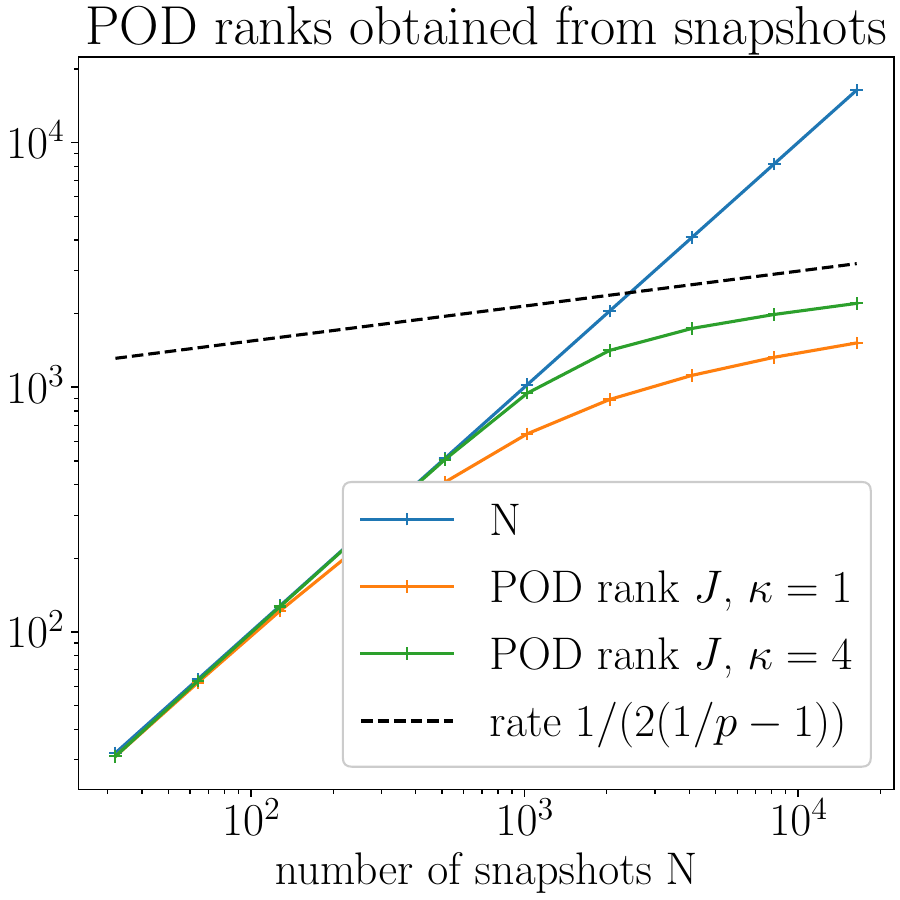}
    \caption{Decay of the eigenvalues of the covariance operator \cref{eq:surfacecovarianceoperator} used to generate the domain deformations (left) and resulting POD-ranks for wavenumbers $\kappa=1$ and $\kappa=4$ (right).}
    \label{fig:eigenvaluedecay}
\end{figure}

\subsection{Implementation and computational setup}
For the generation of the samples we pursue a black-box approach based on the \texttt{C++} open source software library \texttt{BEMBEL}, see \cite{DHK+20}, which is able to perform the required computations within the isogeometric setting \cite{CHB2009,DHJM2022}. For the generation of the training and test samples we rely on Halton points, implying $\alpha=1-\delta$ for any $\delta>0$ throughout the manuscript, and in particular in \cref{cor:PODrankbound,cor:NNparamconv}.
The sampling process is accelerated by a hybrid \texttt{MPI} and \texttt{OpenMP} parallelization, where each sample is accelerated using \texttt{OpenMP}, and the sampling process is accelerated using \texttt{MPI}. The sampling process is performed on 16 nodes of a cluster, with each node being equipped with two Intel Xeon ``Sapphire Rapids'' 48-core processors with 2.10GHz frequency, making 1'536 cores in total, and one \texttt{MPI} process per node. Once the samples are generated, we perform all remaining computations within \texttt{Python} using the \texttt{scipy} package for linear algebra and the \texttt{PyTorch} package for the neural network computations. For the minimization of the loss functional, we use the \texttt{AdamW} optimizer with \texttt{weight\_decay=1} implemented in the \texttt{PyTorch} package and show the results for a single training run. The computations using \texttt{PyTorch} are carried out on a NVIDIA A100 GPU with 40GB RAM.

\subsection{Sampling}
In the following, we focus on the empirical verification of the a-priori bound of the POD rank from \cref{cor:PODrankbound} and the combined Galerkin-POD NN error estimate from \cref{cor:NNparamconv} in the asymptotics in the number of samples $N$. The asymptotic behavior in the Galerkin error and the truncation error for the encoder have been well investigated in the literature, see, e.g., \cite{CD2013,DKL14,HPS2015}, and a study in these parameters is computationally out of reach at the time of writing this article. Thus, in the following we fix those discretizations. For the domain deformations we use second order, globally continuous B-splines with a total of 4116 degrees of freedom. Following \cite{HPS2015}, using an incomplete pivoted Cholesky decomposition with a tolerance of $10^{-4}$ balances the truncation error with the discretization error, and we obtain a truncated approximate expansion \cref{eq:affine_parametric_representation} with 293 terms, i.e., $s=293$. For the solution of the boundary integral equations we employ second order B-splines with 5376 degrees of freedom and a direct solver. Following these choices, we generate $2^{15}=32\,768$ samples for the wavenumbers $\kappa=1$ and $\kappa=4$. For each wavenumber this takes about 34 hours on the 16 nodes of the above mentioned cluster. We will use the first half of these samples for our numerical studies and the second half to measure the POD and $NN$-generalization error.

\subsection{POD-ranks and generalization error}
We determine the POD-reduced basis by truncating the SVD of the snapshot matrix with respect to the Euclidean inner product such that the error in the Frobenius norm is below a certain tolerance $\tau$.
In accordance with \cref{cor:contest}, this tolerance needs to be asymptotically comparable to the Galerkin error, the truncation error for the decoder, and the sampling error of the quasi Monte Carlo sampling. We thus choose $\tau=\frac{1}{100\sqrt{N}}$. The resulting POD-ranks are illustrated in \cref{fig:eigenvaluedecay} and confirm the a-priori bound from \cref{cor:PODrankbound}. \Cref{fig:eigenvaluedecay} also illustrates that more than $10^3$ dimensions are required to capture the essential physics of the considered scattering problem. The predicted POD-error from \cref{cor:PODrankbound} is shown in \cref{fig:N-convergence}.

\subsection{Neural network convergence}
It remains to verify the convergence estimate from \cref{cor:NNparamconv}. To this end, inspired by \cref{cor:NNparamconv}, we choose a tanh NN with depth $\max\{1,\log_2(n)/2\}+2$ and width $n^2$ with $n=N^\frac{p}{(1-p)(2-p)}$ and $p=4/9$, cf.\ \cref{sec:model}. We use the the \texttt{AdamW} optimizer with regularization parameter $\alpha=1$, initial learning rate $10^{-3}$, $\varepsilon=10^{-12}$, and standard settings otherwise. As learning rate scheduler we use \texttt{ReduceLRONPlateau} with standard settings and $\varepsilon=10^{-20}$. In accordance with \cref{eq:stopping_criterion} we stop the iteration when $L_{\text{MSE}}<N^{-\alpha}$. To further stabilize and accelerate the training process we employ BatchNormalization, cf.\ \cite{IS2015}, as provided by the \texttt{PyTorch} package, and normalize the training data to values in $[-1,1]$. The results are illustrated \cref{fig:N-convergence} and seem to indicate that the theoretically predicted convergence rates from \cref{cor:NNparamconv} hold. We attribute the stagnation for larger values of $N$ to the gap between theory and practice discussed in \cref{sec:NN_training}.
\begin{figure}
    \centering
    \includegraphics[height=0.35\textwidth]{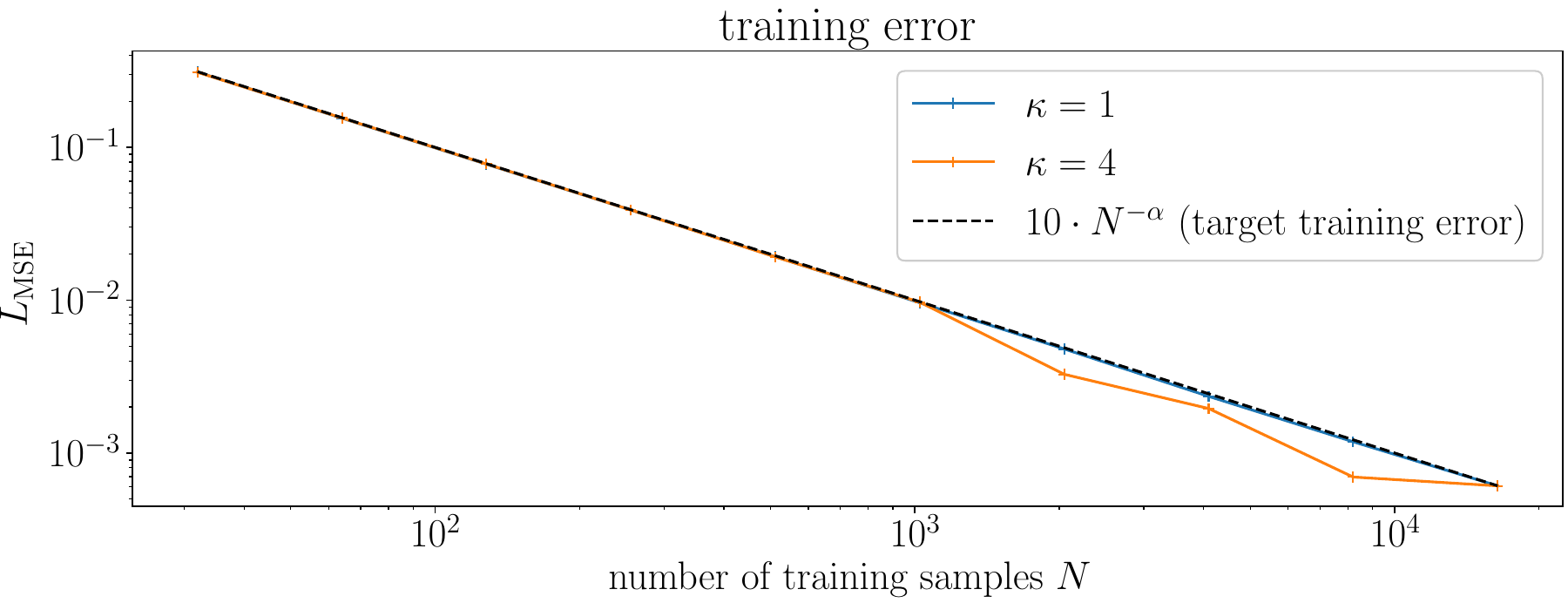}\\
    \includegraphics[height=0.35\textwidth]{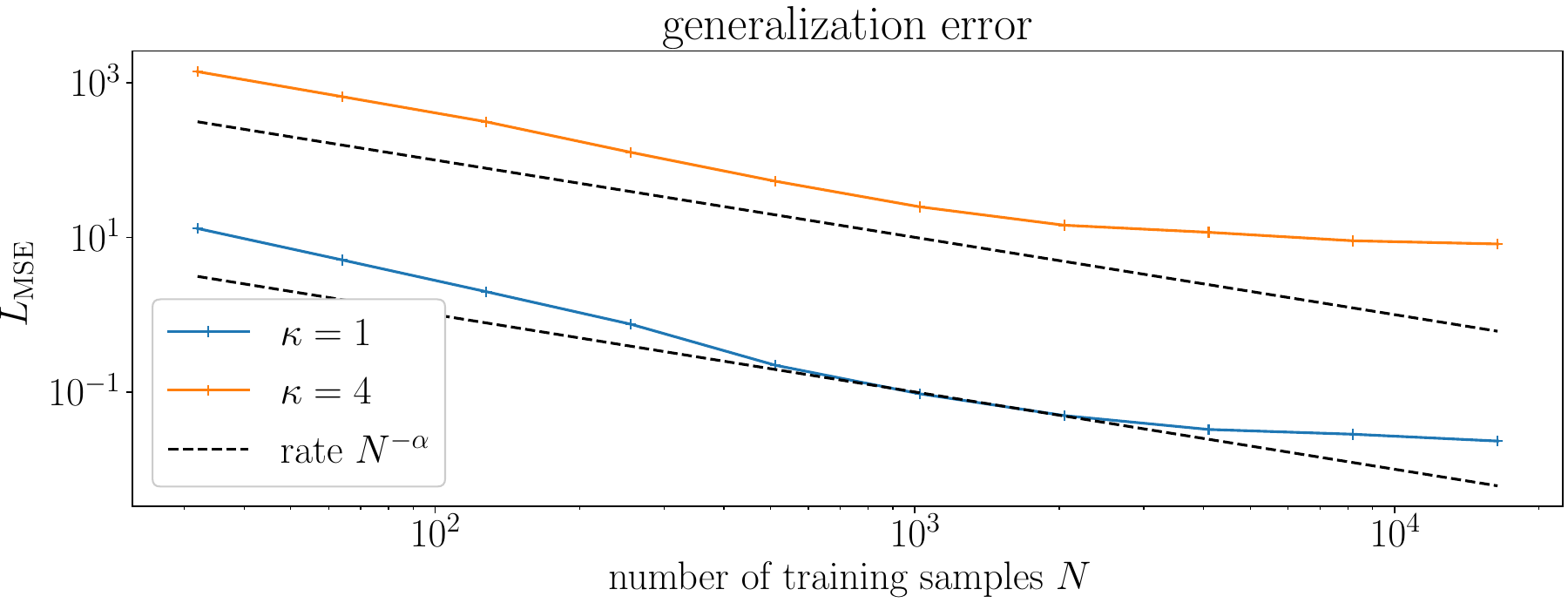}\\
    \includegraphics[height=0.35\textwidth]{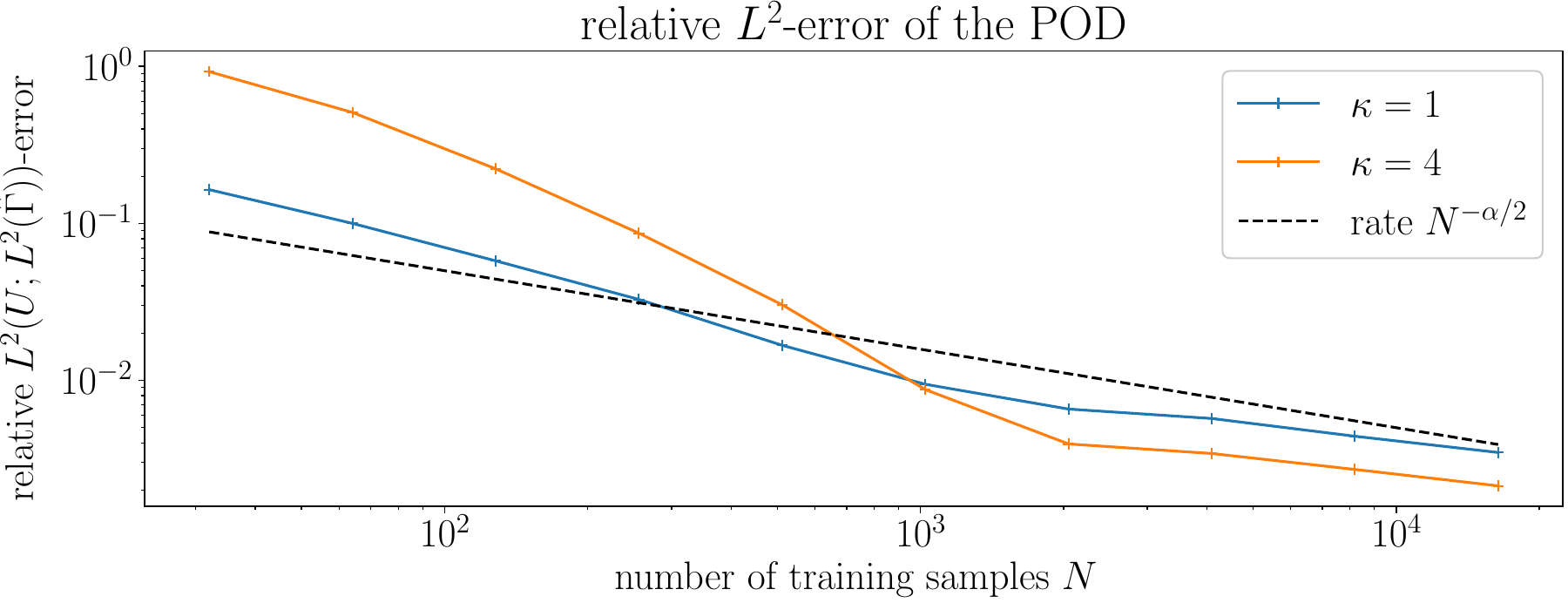}\\
    \includegraphics[height=0.35\textwidth]{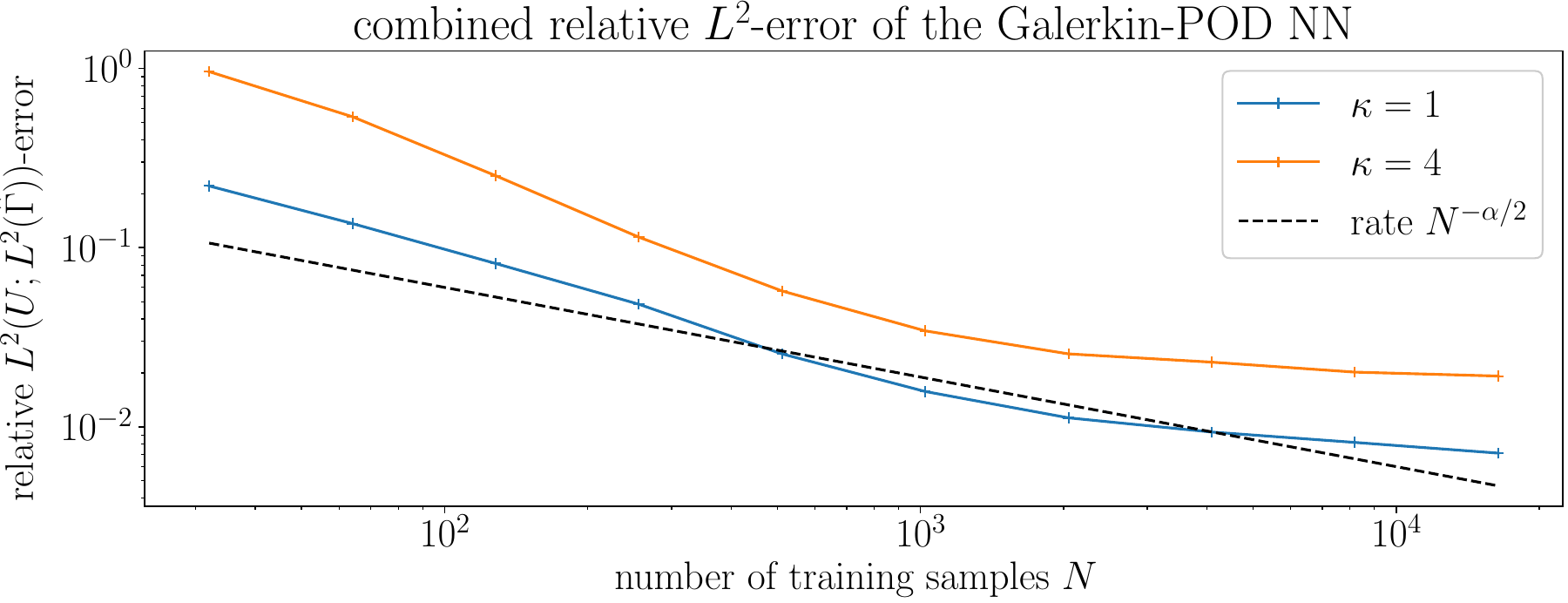}
    \caption{Training (top) and generalization errors (upper middle) for the Galerkin-POD neural network approximation. The $L^2(\mathbb{U};L^2(\widehat{\Gamma}))$-error of the POD (lower middle) confirms the theoretical estimates from \cref{cor:PODrankbound}. The combined Galerkin-POD NN (bottom) seems to confirm the results from \cref{cor:NNparamconv} up to the gap between theory and practice, c.f.~\cref{sec:NN_training}.}
    \label{fig:N-convergence}
\end{figure}
\section{Concluding Remarks}
\label{sec:concluding_remarks}
In this work, we consider the problem of approximating the parameter-to-solution
map associated to parameter-dependent variational problems using the so-called
Galerkin POD-NN method \cite{HU18}. We present a fully discrete error analysis accounting
for a variety of error sources in the construction of a basis of reduced order
by means of the Galerkin POD and discuss how this translates in their approximation 
using NNs in the Galerkin POD-NN. The analysis is applicable to a rather general class of variational problems and yields a-priori estimates on the POD ranks and NN parameters in terms of parametric regularity. Our numerical examples for three-dimensional wave scattering demonstrate that our analysis is applicable to black-box implementations where obtaining the training samples is achieved by a specialized software package, whereas POD and NN computations can be done with standard tools in \texttt{Python}. A remaining issue for practical applications is the NN training, which is based on assumption \cref{eq:MSEassumption}. This assumption, and the fact that the training procedure may get stuck in a local minimum, are the only gaps between theory and practice in our theory. While we did not have issues to reach the required training tolerance by our stopping criterion \cref{eq:stopping_criterion}, our numerical experiments indicate that more consideration and future research needs to be put into addressing \cref{eq:MSEassumption}. Further future work encompasses the use of multi-level NN strategies as the one proposed in \cite{heiss2021neural}, together with extension to time-dependent 
problems and electromagnetic wave scattering.

\section*{Data availability}
The snapshot data and the \texttt{Python} code for the POD and training of the NN are publicly available on the \texttt{bonndata} fileservers \cite{bonndatalink}.

\section*{Acknowledgements}
The authors appreciate access to the Marvin cluster of the University of Bonn and the support of the Deutsche Forschungsgemeinschaft (DFG, German Research Foundation) under Germany's Excellence Strategy -- GZ 2047/1, Projekt-ID 390685813 -- and project 501419255.

FH's work was funded by the Deutsche Forschungsgemeinschaft (DFG, German Research Foundation) – Project-ID 258734477 – SFB 1173 and the Austrian Science Fund (FWF) under the project I6667-N. Funding was also received from the European Research Council (ERC) under the Horizon 2020 research and innovation program of the European Union (Grant agreement No. 101125225).

\appendix
\section{Auxiliary Results}

\begin{lemma}\label{lmm:bpe_holomorphy_pod}
Let $X$ be a complex Hilbert space equipped with the inner product $\dotp{\cdot}{\cdot}_X$
and induced norm $\norm{\cdot}_X$.
    Let $\mathbb{U} \ni \y \mapsto f(\y) \in X$ be $(\boldsymbol{b},p,\varepsilon)$-holomorphic and continuous map 
	when $\mathbb{U}$ 	is equipped with the product topology.
	Then the map
	\begin{align}\label{eq:normextension}
	\mathbb{U}
	\ni
	\y
	\mapsto
    \norm{
    f(\y)
    }^2_{X}
	\in
	\IR
	\end{align}
	is $(\boldsymbol{b},p,\varepsilon)$-holomorphic
	and continuous with the same
	$\boldsymbol{b} \in \ell^p(\IN)$, $p \in (0,1)$,
	and $\varepsilon>0$.
\end{lemma}

\begin{proof}
	We proceed to verify \cref{def:bpe_holomorphy} item-by-item.
	Firstly, one can readily verify that the uniform boundedness of 
	the map introduced in \cref{eq:normextension}.
	
	Let $\boldsymbol\rho\coloneqq (\rho_j)_{j\geq1}$ 
	be any $(\boldsymbol{b},p,\varepsilon)$-admissible 
	sequence of numbers of numbers strictly larger than one.
   	We consider the complex extension of \cref{eq:normextension} to $\mathcal{O}_{\boldsymbol\rho} $
	given by
	\begin{align}\label{eq:hol_ext_norm}
	\mathcal{O}_{\boldsymbol\rho} 
	\ni
	\z
	\mapsto 
	g(\z)
	\coloneqq
	\dotp{f(\z)}{f(\overline{\z})}_X.
	\end{align}
	Observe that this extension is well-defined for each $\z \in \mathcal{O}_{\boldsymbol\rho} $
	since this straightforwardly implies  $\overline{\z} \in \mathcal{O}_{\boldsymbol\rho} $.

	Computing the complex derivative of $g(\z)$
	for $\z\in \mathcal{O}_{\boldsymbol\rho} $ we obtain
	\begin{equation}
	\begin{aligned}
	\frac{dg}{dz_j}(\z)
	=
	&
	\lim_{\snorm{h} \rightarrow 0^{+}}
	\frac{
		\dotp{f(\z + h \boldsymbol{e}_j)}{f(\overline{\z + h \boldsymbol{e}_j})}_X
		-
		\dotp{f(\z)}{f(\overline{\z})}_X
	}{h} 
	\\
	=
	&
	\lim_{\snorm{h} \rightarrow 0^+}
	\frac{
		\dotp{f(\z + h \boldsymbol{e}_j)}{f(\overline{\z + h \boldsymbol{e}_j})}_X
		-
		\dotp{f(\z)}{f(\overline{\z + h \boldsymbol{e}_j})}_X
	}{h} 
	\\
	&
	+
	\lim_{\snorm{h}\rightarrow 0^+}
	\frac{
		\dotp{f(\z)}{f(\overline{\z + h \boldsymbol{e}_j})}_X
		-
		\dotp{f(\z)}{f(\overline{\z})}_X
	}{h} 
	\\
	=
	&
	\lim_{\snorm{h} \rightarrow 0^+}
	\dotp{
		\frac{
			f(\z + h \boldsymbol{e}_j)-f(\z)
		}{h}
	}{f(\overline{\z + h \boldsymbol{e}_j})}_X
	\\
	&
	+
	\lim_{\snorm{h} \rightarrow 0^+}
	\dotp{
		f(\z)
	}
	{
		\frac{
			f(\overline{\z+h \boldsymbol{e}_j})-f(\overline{\z})
		}{
			\overline{h}
		}
	}_X.
	\end{aligned}
	\end{equation}
	Exploiting the continuity of the inner product in each argument and that it is
	anti-linear in the second argument, yields
\begin{equation}
\begin{aligned}
	\frac{dg}{dz_j}(\z)
	=
	&
	\dotp{
		\lim_{\snorm{h} \rightarrow 0^+}
		\frac{
			f(\z+h \boldsymbol{e}_j)-f(\z)
		}{h}
	}{f(\overline{\z})}_X
	\\
	&
	+
	\dotp{
		f(\z )
	}
	{
		\lim_{\snorm{h}\rightarrow 0^+}
		\frac{
			f(\overline{\z+h \boldsymbol{e}_j})-f(\overline{\z})
		}{
			\overline{h}
		}
	}_X.
\end{aligned}
\end{equation}
	Observing that
	\begin{equation}\label{eq:complex_der_norm}
	\lim_{\snorm{h}\rightarrow 0^+}
	\frac{
		f(\overline{\z+h \boldsymbol{e}_j})-f(\overline{\z})
	}{
		\overline{h}
	}
	=
	\lim_{\snorm{h}\rightarrow 0^+}
	\frac{
		f(\overline{\z}+h\boldsymbol{e}_j)-f(\overline{\z})
	}{
		h
	}
	=
	\frac{df}{dz_j}(\overline{\z}),
	\end{equation}
	implies
	\begin{align}
	\frac{dg}{dz_j}(\z)
	=
	\dotp{\frac{df}{dz_j}(\z)}{f(\overline{\z})}_X
	+
	\dotp{f(\z)}{\frac{df}{dz_j}(\overline{\z}))}_X,
	\quad
	\z \in \mathcal{O}_{\boldsymbol\rho}.
	\end{align}
	Observe that, similarly as with \cref{eq:hol_ext_norm},
	the expression in \cref{eq:complex_der_norm} is well-defined
	for any $\z \in \mathcal{O}_{\boldsymbol\rho} $
	since this implies $\overline{\z} \in \mathcal{O}_{\boldsymbol\rho} $.
\end{proof}

\section{Quasi-Monte Carlo Integration}\label{sec:QMC}

Aiming to compute integrals of the kind
\begin{align}\label{eq:integral}
\mathcal{I}(f)
=
\int\limits_{\mathbb{U}}
f(\y) \mu(d\y),
\end{align}
for continuous $f\colon\mathbb{U}\to\mathbb{R}$, we perform a domain truncation from infinite dimensions to the finite dimensional setting.
To this end, let $s\in \IN$ be the truncation dimension and set $\mathbb{U}^{(s)}\coloneqq [-1,1]^s$.
This allows to approximate \cref{eq:integral} by a numerical approximation of an integral
over $\mathbb{U}^{(s)}$ of the form
\begin{align}
\mathcal{I}^{(s)}(f)
\coloneqq  
\int\limits_{\mathbb{U}^{(s)}}
f(\y) \mu^{(s)}(d\y),
\end{align}
where $f^{(s)}\colon\mathbb{U}^{(s)}\to\mathbb{R}$ and $f^{(s)}(y_1,\ldots,y_s)=f(y_1,\ldots,y_s,0,\ldots)$. One of the most common ways to evaluate the integral are Monte Carlo-type quadrature rules with equal weights of the form
\begin{align}\label{eq:MC}
\mathcal{Q}^{(N,s)}(f)
\coloneqq 
\frac{1}{N}
\sum_{n=1}^{N} f\left(2 \y^{(n)}-1\right),
\end{align}
where $ f\left(2 \y^{(n)}-1\right)$ corresponds to the evaluation of the integrand in sampling points $\{\y^{(1)},\dots, \y^{(N)}\} \subset [0,1]^s$. Although plain vanilla Monte Carlo methods are well known to lead to slow convergence in the root mean square sense, quasi-Monte Carlo methods provide faster and rigorous convergence rates for $(\boldsymbol{b},p,\varepsilon)$-holomorphic integrands. In the following, we recall approximation estimates for the Halton sequence and IPL sequences.
\begin{lemma}[{Adaptation of \cite[Lemma 7]{HPS16}}]\label{lemma:Halton}
    Let $s\geq1$ and $\boldsymbol{\beta} = \{\beta_j\}_{j\in \mathbb{N}}$ be a positive number sequence, and let $\boldsymbol{\beta}_s = \{\beta_j\}_{j=1}^{s}$ denote the first $s$ terms.
	Assume that $\boldsymbol{\beta} \in \ell^p(\mathbb{N})$ for some $p\in(0,\frac{1}{3})$.
    Consider the function $f: \mathbb{U} 
    \rightarrow \mathbb{R}$ and assume that
    \begin{align}\label{eq:bound_param_der_1}
	\snorm{\left(\partial^{\boldsymbol{\nu}}_\y f\right)(\y)} 
	\leq 
	c\snorm{\boldsymbol{\nu}}! \boldsymbol{\beta}^{\boldsymbol\nu}_s, 
	\quad
	\text{for all }
	\boldsymbol{\nu} \in \mathbb{N}^s, 
	\quad 
	s\in \mathbb{N},
	\end{align}
    Assume that the sample points in \cref{eq:MC} are drawn according to the Halton sequence.
	Then
	\begin{equation}
	\snorm{
		\mathcal{I}^{(s)}(f)
		-
		\mathcal{Q}^{(N,s)}(f)
	}
	\leq
	C(\delta) N^{\delta-1},
	\end{equation}
	where $C(\delta) \rightarrow \infty$ as $\delta \rightarrow 0$.
\end{lemma}

As explained previously, IPL rules have been proven to deliver convergence rates  
that are independent of the underlying parametric dimension, provided  that the integrand satisfies
specific parametric regularity estimates. The following result addresses this issue.

\begin{lemma}[{\cite[Theorem 3.1]{DKL14}}]\label{prop:QMC_error}
	For $m\geq1$ and a prime $\normalfont\text{b}$, let $N= \normalfont\text{b}^m$ denote the number of HoQMC points.
	Let $s\geq1$ and $\boldsymbol{\beta} = \{\beta_j\}_{j\in \mathbb{N}}$ be a positive number sequence, and let $\boldsymbol{\beta}_s = \{\beta_j\}_{j=1}^{s}$ denote the first $s$ terms.
	Assume that $\boldsymbol{\beta} \in \ell^p(\mathbb{N})$ for some $p\in(0,1)$.
	If there exists $c>0$ such that a function $f: \mathbb{U} \rightarrow \mathbb{R}$ satisfies for $\alpha\coloneqq  \left\lfloor \frac{1}{p} \right\rfloor +1$ that
	\begin{align}\label{eq:bound_param_der}
	\snorm{\left(\partial^{\boldsymbol{\nu}}_\y f\right)(\y)} 
	\leq 
	c\snorm{\boldsymbol{\nu}}! \boldsymbol{\beta}^{\boldsymbol\nu}_s, 
	\quad
	\text{for all }
	\boldsymbol{\nu} \in \{0,1,\dots,\alpha\}^s, 
	\quad 
	s\in \mathbb{N},
	\end{align}
	then the interlaced polynomial lattice rule of order $\alpha$ with $N$ points can be constructed in 
	\begin{align}
	\mathcal{O}\left(\alpha s N \log N + \alpha^2 s^2N\right)
	\end{align}
	operations, such that for the quadrature error holds
	\begin{align}
	\snorm{\mathcal{I}^{(s)}(f) - \mathcal{Q}^{(N,s)}(f)} 
	\leq 
	C_{\alpha,\boldsymbol{\beta},b,p} N^{-1/p}
	\end{align}
	where the constant $C_{\alpha,\beta,b,p}<\infty$ is independent of $s$ and $N$.
\end{lemma}

\begin{remark}
Concerning \Cref{lemma:Halton,prop:QMC_error}, we highlight the following.
\begin{itemize}
    \item 
    To be precise, in \cite{HPS16},
    through a multivariate differentiation argument, it is shown that for a parametric elliptic problem arising from the so-called domain mapping approach, the parametric derivatives of the solution satisfy a bound of the form presented \cref{eq:bound_param_der_1}. This, in turn, is the key step to prove in \cite[Lemma 7]{HPS16} convergence of the Halton quadrature rule. An inspection of the proof reveals that the statement is valid as long as the 
    derivative bounds hold regardless of the underlying model problem. 
    \item
    As pointed out in \cite[Theorem 3.1]{DLS2016}, provided that a parametric map $\mathbb{U} \ni \y \mapsto f(\y) \in \mathbb{R}$ is 
$(\boldsymbol{b},p,\varepsilon)$-holomorphic
for some $\boldsymbol{b} \in \ell^p(\IN)$ for some $p \in (0,1)$ and $\varepsilon>0$ then the parametric multivariate
derivatives satisfy \cref{eq:bound_param_der_1,eq:bound_param_der}. Consequently, the $(\boldsymbol{b},p,\varepsilon)$-holomorphy property is the key to unlock the dimension-independent convergence rates stated in \cref{lemma:Halton,prop:QMC_error}.
\end{itemize}
\end{remark}
\bibliographystyle{siamplain}
\bibliography{ref2}

\end{document}